\documentclass[11pt,letterpaper]{amsart}

\usepackage{amsmath,amssymb,amsthm,verbatim, csquotes,mathrsfs,stmaryrd}
\usepackage{hyperref}
\usepackage{xcolor}
\usepackage{esint}
\usepackage{mathtools}
\usepackage{graphicx}
\usepackage{float}
\usepackage{caption}
\mathtoolsset{showonlyrefs}
\usepackage{geometry}

\newtheorem{theorem}{Theorem}[section]
\newtheorem{lemma}[theorem]{Lemma}
\newtheorem{proposition}[theorem]{Proposition}
\newtheorem{definition}[theorem]{Definition}
\newtheorem{corollary}[theorem]{Corollary}
\newtheorem{remark}[theorem]{Remark}

\def\Om{\Omega}
\def\p{\partial}

\def\de{\delta}
\def\De{\Delta}

\def\S{{\Sigma}}
\def\<{\langle}
\def\>{\rangle}
\def\div{{\rm div}}
\def\na{\nabla}

\def\spt{{\rm spt}}

\providecommand{\abs}[1]{\lvert#1\rvert}
\providecommand{\Abs}[1]{\ensuremath{\left\lvert#1\right\rvert}}

\providecommand{\norm}[1]{\lVert#1\rVert}

\newcommand{\mbN}{\mathbb{N}}

\newcommand{\mbR}{\mathbb{R}}
\newcommand{\mbS}{\mathbb{S}}

\newcommand{\mcA}{\mathcal{A}}
\newcommand{\mcB}{\mathcal{B}}

\newcommand{\mcH}{\mathcal{H}}

\newcommand{\mcL}{\mathcal{L}}
\newcommand{\mcM}{\mathcal{M}}

\newcommand{\mcR}{\mathcal{R}}
\newcommand{\mcS}{\mathcal{S}}

\newcommand{\rd}{{\rm d}}

\newcommand{\wsc}{\overset{\ast}{\rightharpoonup}}

\newcommand{\ra}{\rightarrow}

\newcommand{\eq}[1]{\begin{equation}\begin{alignedat}{2} #1 \end{alignedat}\end{equation}}

\numberwithin{equation} {section}

\begin{document}

	
\title[Bernstein theorem]{Bernstein theorem for $3$-dimensional anisotropic minimal graphs with free boundary}
\date{\today}

\author[Wang]{Guofang Wang}
\address[G.W]{Mathematisches Institut\\
	Universit\"at Freiburg\\
	Ernst-Zermelo-Str.1\\
	79104\\Freiburg\\ Germany}
\email{guofang.wang@math.uni-freiburg.de}

\author[Wei]{Wei Wei}
\address[W.W]{School of Mathematics\\ Nanjing University\\ 210093\\Nanjing\\ P.R. China}
\email{wei\_wei@nju.edu.cn}
\thanks{W.W. is supported by NSFC (Grant No.~12571218, 12271244) and Alexander von Humboldt fellowship}

\author[Xia]{Chao Xia}
\address[C.X]{School of Mathematical Sciences\\
	Xiamen University\\
	361005, Xiamen, P.R. China}
\email{chaoxia@xmu.edu.cn}
\thanks{C.X. is supported by NSFC (Grant No. 12271449, 12671069, 12526203, 12526102) and the Natural Science Foundation of Fujian Province of China (Grant No. 2024J011008).}

\author[Zhang]{Xuwen Zhang}
\address[X.Z]{Mathematisches Institut\\
	Universit\"at Freiburg\\
	Ernst-Zermelo-Str.1\\
	79104\\Freiburg\\ Germany}
\email{xuwen.zhang@math.uni-freiburg.de}


\begin{abstract}
In this paper, we continue our recent study on anisotropic minimal surface equation with free boundary condition in Wang-Wei-Xia-Zhang (Arch Ration Mech Anal 250:42, 2026). The first main result solves the corresponding mixed boundary value problem. As an application, we prove the following Bernstein-type theorem: any anisotropic minimal graph over $\mathbb{R}^3_+$ with free boundary must be flat. These extend the classical results of Simon (Indiana Univ Math J 25:821--855, 1976; Math Z 154:265--273, 1977) to an appropriate free boundary setting.
\\

\noindent {\bf MSC 2020: 53A10,  35J93, 35J25}\\
{\bf Keywords:} Anisotropic minimal graph, free boundary condition, mixed boundary problem, Bernstein theorem\\
\end{abstract}

\maketitle
\tableofcontents

\section{Introduction}
Let $F:\mbR^{n+1}\to\mbR_+$ be a smooth, uniformly elliptic integrand. Namely, $F$ is a one-homogeneous function on $\mbR^{n+1}$ which is smooth and positive and on $\mbS^n$, satisfying that the Wulff shape $\{F<1\}$ is uniformly convex.
Associated to $F$ we consider the function $f$ on the hyperplane $\mbR^n=\{x_{n+1}=0\}$ defined by
\eq{
f(y)
\coloneqq F(-y,1),\quad y\in\mbR^n.
}

Let $u\in C^2(\Om)$, where $\Om$ is a domain in $\mbR^n\coloneqq\{x_{n+1}=0\}$. We say that $u$ satisfies the {\em anisotropic minimal surface equation} (AMSE) if
\eq{\label{defn:AMSE}
{\rm div}\bigl(Df(Du(x))\bigr)=0\quad\text{on }\Om,
}
where $D$ and ${\rm div}$ denote the Euclidean gradient and divergence on $\mbR^n$. When $\p\Om$ is non-empty and sufficiently regular, we say that $u$ satisfies the {\em anisotropic free boundary condition} if
\eq{\label{defn:anisotropic-free-bdry-Intro}
\langle Df(Du(x)),\bar N\rangle=0,\quad x\in\p\Om,
}
where $\bar N$ is the outer unit normal of $\Om\subset\mbR^n$ along $\p\Om$.

Equations \eqref{defn:AMSE}--\eqref{defn:anisotropic-free-bdry-Intro} are the Euler--Lagrange conditions for the anisotropic surface energy among graphs over $\Om$. Geometrically, if $M=\{(x,u(x)):x\in\Om\}$ is the graph of $u$, then $M\subset\Om\times\mbR$ is an anisotropic minimal hypersurface and, along $\p M\subset\p\Om\times\mbR$, its anisotropic normal $\nu_F\coloneqq \bar D F(\nu)$ (the {\em Cahn--Hoffman map} applied to the Euclidean unit normal $\nu$) satisfies
\eq{\label{eq:<nu_F,e_1>=0}
\langle \nu_F, \bar N\rangle =0.
}
When $F=F_{eucl}$, \eqref{defn:AMSE} reduces to the classical \emph{minimal surface equation} (MSE)
\eq{\label{defn:MSE}
{\rm div}\left(\frac{Du}{\sqrt{1+\abs{Du}^2}}\right)=0,
}
and \eqref{defn:anisotropic-free-bdry-Intro} becomes the usual orthogonality condition $\langle \nu,\bar N\rangle=0$.

The study of MSE and AMSE has a long history, here we briefly recall some results related to our work.
For the classical minimal surface equation, Bombieri-De Giorgi-Miranda \cite{BDeGM69} established the interior gradient estimates ($n=2$ already by Finn \cite{Finn54}) by using the nowadays well-known integral methods. They proved that
\eq{\label{eq:BDG-gradient-estimate}
\abs{Du(x)}
\leq\exp\left[C(n)+\frac{C(n)}r\left(\sup_{B^n_r(x)}u-u(x)\right)\right],\quad\forall r>0\text{ whenever }B^n_r(x)\subset\Om.
}
It is well-known that mean curvature type equations are not uniformly elliptic before a gradient estimate is known, therefore \eqref{eq:BDG-gradient-estimate} is central in the study of MSE.

Among many important applications of \eqref{eq:BDG-gradient-estimate}, we list two of them: the first one is the so-called Liouville theorem \cite{BDeGM69}, concerning the global rigidity of MSE, which states that if $\Om=\mbR^n$ and $u$ has one-sided linear growth, then $u$ must be affine.
The second one is the solvability of the Dirichlet problem for \eqref{defn:MSE}, which states that if $\Om\subset\mbR^n$ is a bounded open domain with $\p\Om\in C^2$, and $\varphi\in C^0(\p\Om)$, then there exists a unique $u\in C^0(\overline\Om)\cap C^2(\Om)$ solving \eqref{eq:<nu_F,e_1>=0} with $u=\varphi$ on $\p\Om$ if and only if $\p\Om$ is mean convex.
The result was proved by Jenkins-Serrin \cite{JS68} provided $\varphi\in C^{2,\gamma}(\p\Om)$.
Combining \cite{JS68} with \eqref{eq:BDG-gradient-estimate}, one can carry out a compactness argument and relax $\varphi$ to be only continuous, see \cite[Chapter 16]{GT01}.

For a general uniformly elliptic integrand, Simon \cite{Simon76} obtained the Bombieri-De Giorgi-Miranda-type interior gradient estimates in a very broad class of non-uniformly elliptic equations by the integral method. In terms of AMSE \eqref{defn:AMSE}, Simon's estimates can be simplified (cf. \cite[Section 4]{Simon76}) as \eqref{eq:BDG-gradient-estimate} with $C$ additionally depending on $F$. As an application, he showed in \cite[Section 5]{Simon76} the solvability of the Dirichlet problem for \eqref{defn:AMSE} with continuous boundary data.
Later in 
\cite{Simon77}, Simon used his interior gradient estimate and the solvability of the Dirichlet problem to show the removability of singularities of the AMSE \eqref{defn:AMSE} and, more importantly, a Bernstein theorem for entire anisotropic minimal graph over $\mbR^3$, which we will discuss in detail below.
We also note that his gradient estimate yields readily the Liouville-type theorem for entire anisotropic minimal graph by a standard argument.

We now turn to the free boundary problem \eqref{defn:AMSE}-\eqref{defn:anisotropic-free-bdry-Intro}.
In our recent work \cite{WWXZ26}, we established the following free boundary counterpart of the gradient estimates of Bombieri-De Giorgi-Miranda \cite{BDeGM69} and Simon \cite{Simon76}:
\eq{\label{eq:gradient-esti-WWXZ26}
\abs{Du(x)}
\leq\exp\left[C(n,F)+\frac{C(n,F)}r\left(\sup_{B^{n+}_r(x)}u-u(x)\right)\right],\quad\forall r>0,
}
where $B^{n+}_r(x)=B^n_r(x)\cap\mbR^n_+$ and $\mbR^n_+\coloneqq\{x_1>0:x_{n+1}=0\}$.
Using the standard argument as mentioned above, we prove a Liouville-type theorem for anisotropic minimal graph with free boundary over $\mbR^n_+$, see \cite[Theorem 1]{WWXZ26}.
Motivated by Simon's solvability of Dirichlet problem \cite{Simon76}, it is natural to ask whether \eqref{eq:gradient-esti-WWXZ26} leads to the solvability of the AMSE \eqref{defn:AMSE} with mixed boundary conditions. In this regard, our first main result is as follows.

For $\varrho>0$ fixed, we use the conventions $B^n_\varrho=B^n_\varrho(0)$, $B^{n+}_\varrho=B^{n+}_\varrho(0)$, and also
\eq{\label{defn:conventions-Om}
 \Omega_\rho=B^{n+}_{\varrho},
 \quad
 \Gamma_\rho=B^n_\varrho\cap\p\mbR^n_+,
 \quad
 S_\varrho=\p B^n_\varrho\cap\overline{\mbR^n_+},\quad
 E_\varrho
 =\overline\Gamma_\varrho\cap\overline S_\varrho.
}

\begin{theorem}\label{thm:replacement}
Let $n\geq2$. Suppose $\phi$ is the restriction to $\overline S_\varrho$ of a $C^2$ function defined in a neighborhood of $\overline S_\varrho$. Then there exists a unique $w\in W^{1,1}(\Omega_\varrho)\cap C^2(\Omega_\varrho\cup\Gamma_\varrho)$ such that
\eq{\label{eq:replacement-PDE}
 \begin{cases}
 {\rm div}(Df(Dw))=0 & \text{in }\Om_\varrho,\\
 \langle Df(Dw),e_1\rangle=0 & \text{on }\Gamma_\varrho,
 \end{cases}
}
and
\eq{\label{eq:replacement-trace}
 \lim_{\Om_\varrho\ni x\to y}w(x)=\phi(y)
 \quad\text{for every }y\in S_\varrho\setminus E_\varrho.
}
Moreover, for every $\zeta\in C^1(\overline\Om_\varrho)$ with $\zeta=0$ on
$S_\varrho$ and $\operatorname{spt}\zeta\cap E_\varrho=\emptyset$,
\eq{\label{eq:replacement-weak}
 \int_{\Om_\varrho} \left<Df(Dw),D\zeta\right>\rd x
 =0.
}
\end{theorem}
In the isotropic case, mixed Dirichlet-capillary problems for the prescribed mean curvature equation were studied by Giusti \cite{Giusti75,Giusti76} and Gerhardt \cite{Gerhardt79}, who obtained existence and regularity results for the solutions.
Obersnel and Omari \cite{OO09,OO25} studied mixed Dirichlet-Neumann and Dirichlet-capillary problems in the setting of $BV$ functions.
For general uniformly elliptic equations, Lieberman \cite{Liebermann86,Liebermann89,Lieberman13} developed regularity
theory for mixed Dirichlet-oblique boundary problems, including regularity at the corner.
To our knowledge, there is less progress for the anisotropic mean curvature type equation, and we expect the establishment of Theorem \ref{thm:replacement} to have potential in many physical or geometrical applications in the future.


As one application, we obtain the following removability result at the free boundary.

\begin{corollary}\label{Prop:removable-singularity}
Let $s>0$, suppose that $K\subset\subset\Om_s\cup\Gamma_s$ is compact and satisfies $\mcH^{n-1}(K)=0$.
Let $u\in C^2(\left(\Om_s\cup\Gamma_s\right)\setminus K)$ satisfy
\eq{
\begin{cases}
\div(Df(Du))=0 &\text{in }\Om_s\setminus K\\
\left<Df(Du),e_1\right>=0 &\text{on }\Gamma_s\setminus K.
\end{cases}
}
Then $u\in C^{2}(\Om_s\cup\Gamma_s)$.
\end{corollary}

In the spirit of Simon \cite{Simon77}, we use Corollary \ref{Prop:removable-singularity} to prove the following {\em half-space anisotropic Bernstein theorem} for anisotropic minimal graph over $\mbR^3_+$.

\begin{theorem}\label{Thm:Bernstein-aniso}
Let $F:\mbR^{n+1}\ra\mbR_+$ be a smooth, uniformly elliptic integrand and
$u$  a $C^2$-solution to \eqref{defn:AMSE} on $\mbR^n_+$ satisfying the anisotropic free boundary condition \eqref{defn:anisotropic-free-bdry-Intro}. If $n=3$, then $u$ must be affine.
\end{theorem}

In the isotropic case, the classical {\em Bernstein problem} asks whether an entire solution of \eqref{defn:MSE} on $\mbR^n$ must be affine. This holds for $n\le 7$ (Bernstein \cite{Bernstein27}, Fleming \cite{Fleming62}, Almgren \cite{Almgren66}, De Giorgi \cite{DeGiorgi65}, Simons \cite{Simons68}) and fails for $n\ge 8$ (Bombieri--De Giorgi--Giusti \cite{BDG69}).
For a general elliptic integrand $F$, the (interior) anisotropic Bernstein problem for \eqref{defn:AMSE} has been settled: it holds for $n=2,3$ (Jenkins \cite{Jenkins61}, Simon \cite{Simon77}) and fails for $n\ge 4$ (Mooney \cite{Mooney22}, Mooney--Yang \cite{MY24}). In particular, Mooney--Yang constructed in \cite{MY24} pairs $(u,F)$ solving \eqref{defn:AMSE} on $\mbR^n$ ($n\ge 4$) with superlinear growth
\eq{\label{example:MY24-superlinear}
\sup_{B_r} u\sim r^{1+\mu}
}
for any $\mu\in\bigl(0,\tfrac12\bigr)$. In their $n=4$ example one writes $\mbR^4=\mbR^2\times\mbR^2$ with variables $(x,y)\in\mbR^2\times\mbR^2$, and both $u$ and the induced function $f$ depend only on $\abs{x}$ and $\abs{y}$.

As observed in \cite{WWXZ26}, these examples also yield counterexamples to the {half-space} anisotropic Bernstein problem for $n\ge 4$: the same function $u$ solves \eqref{defn:AMSE} on $\mbR^4_+$ and, by symmetry, satisfies \eqref{defn:anisotropic-free-bdry-Intro} on $\p\mbR^4_+$. On the other hand, when $n=2$ a calibration argument shows that anisotropic minimal graphs minimize the anisotropic surface energy and hence have Euclidean volume growth (see \cite[Section 4.1]{WWXZ26}); therefore the curvature estimate in \cite{GX25} yields the half-space Bernstein theorem for $n=2$.
In view of this, Theorem \ref{Thm:Bernstein-aniso} thus completes the classification for the half-space anisotropic Bernstein problem.
For further rigidity results of the AMSE, we refer interested readers to \cite{EW22,DY24,DMYZ23,WWZ25}; see also \cite{DePDeR24,DePDeRL24,DeRP25,DeRHW26,Yang26} for the recent study of anisotropic minimal surface without graphicality.

As an easy consequence of Theorem \ref{Thm:Bernstein-aniso}, one can reprove the following half-space Bernstein theorem for capillary minimal graphs.
Precisely, choosing $F$ in Theorem \ref{Thm:Bernstein-aniso} to be the so-called capillary gauge (see \cite{PM15,WWXZ26}), defined as $F_\theta(\xi)\coloneqq\abs{\xi}-\cos\theta\left<e_1,\xi\right>$ for any $\xi\in\mbR^{n+1}$, where $\theta\in(0,\pi)$ is a fixed constant characterizing the capillary angle, one immediately obtains:
\begin{corollary}
Let $n=3$. For any $\theta\in(0,\pi)$,
the only  minimal graphs on $\mbR^3_+$ with constant contact angle $\theta$ along $\p\mbR^3_+$,  are half-affine planes.
\end{corollary}
The {\em capillary minimal graphs} have recently attracted considerable attention.
For general $\theta\in(0,\pi)$, the Bernstein theorem for capillary minimal graphs over $\mbR^n_+$ follows from a more general result of Hong--Saturnino \cite{HS23} when $n=2$ (see also \cite{LZZ24,MP21}).
For $n=3$ and $n=4$, it was proved by Edelen--Li--Zhu \cite{ELZ26} and Wang--Zhang \cite{WZ26}.
For $5\le n\le 7$, the corresponding statement holds under additional restrictions on the contact angle $\theta$; see \cite{LZZ24,ELZ26,WZ26}.
It remains an open and challenging problem to determine whether these restrictions can be removed.

\subsection{Idea of the proof}

As said, the main difficulty in proving Theorem \ref{thm:replacement} lies in the fact that \eqref{defn:AMSE} is non-uniformly elliptic unless $\abs{Du}$ is bounded. To solve the mixed boundary problem, we follow the idea of Simon \cite[Section 5]{Simon76} by approximating the anisotropic mixed boundary problem through a family of uniformly elliptic mixed boundary problems as follows: 
for $0<\tau\leq1$, let 
\eq{\label{defn:f_tau}
f_\tau(y)
\coloneqq f(y)+\frac\tau2f^2(y),\quad y\in\mbR^n,
}
and consider 
\eq{\label{eq:uniform elliptic-PDE}
 \begin{cases}
 {\rm div}(Df_\tau(Dw))=0 & \text{in }\Om_\varrho,\\
 \langle Df_\tau(Dw),e_1\rangle=0 & \text{on }\Gamma_\varrho,\\
 \lim_{\Omega_\varrho\ni x\to y}w(x)=\phi(y)
 &  \text{for every }y\in S_\rho\setminus E_\varrho.
 \end{cases}
}
Since $Df_\tau=(1+\tau f)Df$, the free boundary condition of the approximated equations is equivalent to the original one \eqref{defn:anisotropic-free-bdry-Intro}.
Moreover, \eqref{eq:uniform elliptic-PDE} is in fact the Euler-Lagrange equation of the minimization problem of the following functional:
\eq{\label{eq:Atau}
\left\{\mathscr A_\tau(v)
=\int_{\Om_\varrho} f_\tau(Dv)\rd x: v\in W^{1,2}(\Om_\varrho), v|_{{S_\varrho}}
=\phi\right\}.
}
One can show that $f_\tau$ is strictly convex for each $\tau\in(0,1]$, and that there exists a unique minimizer $w_\tau$ of \eqref{eq:Atau}, see Proposition \ref{prop:mixed boundary problem}.
To be able to pass $w_\tau$ to the solution $w$ of \eqref{eq:replacement-PDE} in the spirit of \cite{Simon76}, we need to send $\tau\searrow0$ and have a nice control of a suitable norm of $w_\tau$, independent of $\tau\in(0,1]$.
A key ingredient is the following local estimate independent of the choice of $\tau$ (see Theorem \ref{uniform gradient estimate} below):
\eq{
\sup_{B^{n+}_{R/2}(x)\cup\Gamma_{R/2}(x)}\abs{Dw_\tau}
\leq C(n,F)\exp\left[C(n,F)\left(1+\frac{{\rm osc}_{B^{n+}_{2R}(x)}w_\tau}{R}\right)^{n+3}\right].
}
We will prove this key estimate by adapting the integral method proof of our previous boundary gradient estimate \eqref{eq:gradient-esti-WWXZ26} in \cite{WWXZ26}, which relies crucially on the variational structure of the anisotropic free boundary problem.

We now turn to the proof of Theorem \ref{Thm:Bernstein-aniso}.
It is well-known that an essential difficulty in the study of anisotropic minimal surfaces, in contrast to the isotropic case, is the absence of a monotonicity formula (cf. \cite{Allard74}).
In proving the anisotropic Bernstein theorem, the lack of a monotonicity formula accounts for the fact that, the blow-down limit of the anisotropic minimal graph is no longer a cone. Therefore, the De Giorgi argument \cite{DeGiorgi65,Simons68} cannot be applied directly.

To prove Theorem \ref{Thm:Bernstein-aniso}, we argue in the spirit of Simon's contradiction proof \cite{Simon77}, adapted to the anisotropic free boundary settings. Due to the non-trivial boundary condition \eqref{defn:anisotropic-free-bdry-Intro}, there are mainly two new difficulties that arise.
The first difficulty concerns the removability of singularities of the solution to the AMSE \eqref{defn:AMSE} at the free boundary point.
In the interior case, Simon's argument \cite[Theorem A]{Simon77} uses a regular solution of the Dirichlet problem for replacement. Such a solution exists by the solvability of the Dirichlet problem in \cite[Section 5]{Simon76}.
In our case, our Theorem \ref{thm:replacement} leads to Corollary \ref{Prop:removable-singularity} and solves the problem.

The second difficulty is the establishment of a Harnack-type dichotomy of $\nu_4\coloneqq\left<\nu,e_4\right>$, see \cite[Proof of Lemma 1]{Simon77}, which uses crucially a Harnack estimate for $\nu_4$ shown in Schoen-Simon-Almgren's celebrated work \cite[Lemma 2.7]{SSA77}. We note that the Schoen-Simon-Almgren estimate cannot be directly applied in our case, again mainly due to the non-trivial boundary condition \eqref{defn:anisotropic-free-bdry-Intro}. 
To address the issue, we introduce the quantity $V_F\coloneqq\frac{\nu_4}{F(\nu)}$, and the key observation is that, on anisotropic minimal hypersurface in $\mbR^{n+1}_+$,
\eq{
\Delta_F V_F
=&-2 g\left(\nabla \log F(\nu), \nabla_F\left(V_F\right)\right)-\frac{V_F}{F(\nu)} \operatorname{tr}_g\left(h_F^2\right),
}
together with the boundary condition
$g\left(\nabla_F V_F, \mu\right) \equiv 0$. Here $\mu$ is the classical outer unit conormal along the boundary of a hypersurface, and this in particular shows that $\na_FV_F$ is tangential along the boundary.  See Section \ref{Sec:2} below for the precise notations concerning anisotropy. Consequently, the classical boundary Harnack inequality (cf., \cite[Theorem 5.44]{Lieberman13}) holds for $V_F$, which in turn gives us the required Harnack estimate for $\nu_4$.

\

\noindent{\em The rest of  the paper  is organized as follows.} In Section \ref{Sec:2}, we recall some useful facts on elliptic integrand and anisotropic geometry and collect some basic facts from geometric measure theory.
In Section \ref{Sec:3}, we show the existence of a sufficiently regular solution to the mixed boundary problem (Theorem \ref{thm:replacement}).
In Section \ref{sec:4} we prove the removability of singularities of AMSE at the free boundary points (Corollary \ref{Prop:removable-singularity}).
In Section \ref{Sec:5}, we prove Theorem \ref{Thm:Bernstein-aniso}.

\

\noindent{\em Acknowledgments.}
This work was partly carried out while W. Wei was visiting University of Freiburg supported by the Alexander von Humboldt research fellowship, and while X. Zhang was visiting Xiamen University.
The authors would like to thank Ernst Kuwert for bringing the reference \cite{Giusti76} into our attention, and thank Yuchen Bi for many helpful discussions.

\

\noindent{\em On the use of AI.}
ChatGPT 5.6 Pro assisted the authors in identifying the ansatz \eqref{defn:f_tau}, and in obtaining the estimates needed for the existence and regularity of the mixed boundary problem (Theorem \ref{thm:replacement}).
The authors  verified all arguments and take full responsibility for the completeness and correctness of the proofs.

\section{Preliminaries}\label{Sec:2}
Let $\bar D, \bar{\rm div},\left<\cdot,\cdot\right>$ denote the Euclidean gradient, divergence, and scalar product on $\mbR^{n+1}$, let $D,\,\, {\rm div}$ denote the Euclidean gradient, divergence on $\mbR^n$.
We denote by $B_r(X)$ the Euclidean open ball of radius $r$ in $\mbR^{n+1}$ with center at $X\in\mbR^{n+1}$, and adopt the shorthand $B_r=B_r(0)$.
We also denote by $B_r^n(x)$ the Euclidean open ball of radius $r$ in $\mbR^n=\{x_{n+1}=0\}$, with center $x\in\mbR^n$.
For balls truncated by the half-space $\mbR^{n+1}_+$ (or $\mbR^n_+$), we use the notations $B^+_r(X)=B_r(X)\cap\mbR^{n+1}_+$ (or $B^{n+}_r(X)=B^n_r(X)\cap\mbR^n_+$) and the shorthand $B^+_r=B^+_r(0)$.

We adopt the following notations when considering the topology of $\mbR^{n+1}$: we denote by $\overline{E}$ the topological closure of $E\subset\mbR^{n+1}$, by $\p E$ the topological boundary of $E$, and by $E\De F$ the difference of two sets $E,F\subset\mbR^{n+1}$.
In terms of the subspace topology (relative topology), let $A$ be a topological space and $S$ be a subspace of $A$, we use ${\rm cl}_A(S)$ to denote the closure of $S$ in the topological space $A$.

\subsection{Elliptic integrand}

Let $F:\mbR^{n+1}\ra\mbR_+$ be a smooth, uniformly elliptic integrand.
Precisely, $F\in C^{\infty}(\mbR^{n+1}\setminus\{0\})$ and is positive, one-homogeneous function on $\mbR^{n+1}$ such that $\bar D^2 F^2(z)$ is positive definite for all $z\ne 0$.

We record the following useful properties of $F$:
\begin{itemize}
    \item For any $z\in\mbR^{n+1}\setminus\{0\}$,
    \eq{\label{eq:<DF(z),z>=F(z)}
    \left<\bar DF(z),z\right>=F(z).
    }
    \item Since $F$ is 1-homogeneous, it is well-known that
\eq{\label{eq:Hessian-F}
\bar D^2F(z)(\cdot,z)=0,\quad\forall z\in\mbR^{n+1}\setminus\{0\}.
}
\end{itemize}
The properties of $F$ yield the following standard estimates.
\begin{lemma}\label{lem:structure of f}
There are positive constants $c(F)$ and $C(F)$ such that, for every
$p,\xi\in\mbR^n$,
\begin{align}
 c(F)\sqrt{1+|p|^2}&\le f(p)
 \le C(F)\sqrt{1+|p|^2},\label{eq:f-growth}\\
 |Df(p)|&\le C(F),\label{eq:Df-bound}\\
 \frac{c(F)}{\sqrt{1+|p|^2}}
 \left(|\xi|^2-\frac{\langle p,\xi\rangle^2}{1+|p|^2}\right)
 &\le f_{ij}(p)\xi_i\xi_j
 \le\frac{C(F)}{\sqrt{1+|p|^2}}
 \left(|\xi|^2-\frac{\langle p,\xi\rangle^2}{1+|p|^2}\right).\label{eq:D2f-global}
\end{align}
In particular, $D^2f(p)$ is positive definite for every finite $p$, and $Df$ is strictly monotone:
\eq{\label{eq:strict-monotone}
 \langle Df(p)-Df(q),p-q\rangle>0
 \quad\text{whenever }p\neq q.
}
\end{lemma}
\begin{proof}
By the $1$-homogeneity of $F$ we have
\eq{
c(F)\abs{z}
\leq F(z)
\leq C(F)\abs{z},\quad\forall z\in\mbR^{n+1},
}
and that $\bar DF$ is zero-homogeneous and hence bounded on $\mbS^n$.
By the fact that $Df(p)=-(\bar D_1F,\ldots,\bar D_nF)(-p,1)$ for any $p\in\mbR^n$, we deduce \eqref{eq:f-growth} and \eqref{eq:Df-bound}.

To show \eqref{eq:D2f-global}, we first note that by \eqref{eq:Hessian-F} and the $-1$-homogeneity of $\bar D^2F$, one has: for any fixed $z\neq0\in\mbR^{n+1}$ and any $Y\in\mbR^{n+1}$,
\eq{
\frac{c(F)}{\abs{z}}\abs{Y^\perp}^2
\leq\bar D^2F(z)[Y,Y]
\leq\frac{C(F)}{\abs{z}}\abs{Y^\perp}^2,
}
where $Y^\perp=Y-\frac{\left<Y,z\right>}{\abs{z}^2}z$ is the component of $Y$ orthogonal to $z$.
Letting $z=(-p,1)$, $Y=(-\xi,0)$ for any $p,\xi\in\mbR^n$, and noting that
\eq{
D^2f(p)[\xi,\xi]
=\bar D^2F(z)[Y,Y],
}
we deduce \eqref{eq:D2f-global}.
Finally, integrating the positive
quadratic form $D^2f$ along the segment from $q$ to $p$ gives
\eqref{eq:strict-monotone}.
This completes the proof.

\end{proof}

\begin{lemma}\label{lem:normal controlled by tangential}
There exists $C(F)>0$ such that for  every $p=(p_1,p')\in\mbR^n$ satisfying $\left< Df(p),e_1\right>=0$, there holds 
\eq{\label{eq:transversality}
 \abs{p_1}
 \leq C(F)\sqrt{1+\abs{p'}^2}.
}
\end{lemma}
\begin{proof}
By \eqref{eq:<DF(z),z>=F(z)}, we have
\eq{
 \bar D_1F(e_1)=F(e_1)>0,
 \quad
 \bar D_1F(-e_1)=-F(-e_1)<0.
}
By continuity on the unit sphere, there is $\delta=\delta(F)\in(0,1)$ such
that
\eq{\label{positive part of D1F}
 \bar D_1F>0\quad\text{on }B_\delta(e_1)\cap\mathbb S^n,
 \quad
 \bar D_1F<0\quad\text{on }B_\delta(-e_1)\cap\mathbb S^n.
}

Now let $p$ satisfy $\left< Df(p),e_1\right>=0$ and set $\nu=\frac{(-p,1)}{\sqrt{1+|p|^2}}$.
Since $\langle Df(p),e_1\rangle=-\bar D_1F(-p,1)$ and $\bar DF$ is
zero-homogeneous, we have $\bar D_1F(\nu)=0$. 
It follows from \eqref{positive part of D1F} that $\nu$ lies outside $B_\delta (e_1)$ and $B_{\delta}(-e_1)$. If $\nu_1\geq0$, then
\eq{
 \de^2
 \leq\abs{\nu-e_1}^2
 =2(1-\nu_1),
}
whereas if $\nu_1\leq0$, then
\eq{
 \de^2
 \leq\abs{\nu+e_1}^2
 =2(1+\nu_1).
}
Thus in both cases, we have $\abs{\nu_1}\leq1-\frac{\de^2}{2}$.
Moreover, by the definition of $\nu$, 
\eq{
 \frac{\abs{p_1}}{\sqrt{1+\abs{p'}^2}}
 =\frac{\abs{\nu_1}}{\sqrt{1-\nu_1^2}},
}
which readily implies \eqref{eq:transversality}.

\end{proof}

\subsection{\texorpdfstring{$F$}{F}-minimizers}

Let $A\subset\mbR^{n+1}$ be open.
We call a Lebesgue measurable set $E\subset\mbR^{n+1}$ a \emph{set of locally finite perimeter in $A$} if, for every compact set $K\subset A$,
\eq{
\sup\left\{\int_E\bar{{\rm div}}T\rd X: T\in C_c^1(A;\mbR^{n+1}),\quad\spt T\subset K,\quad\sup_A|T|\leq1
\right\}<\infty.
}
It follows that the characteristic function $\chi_E\in BV_{loc}(A)$. Moreover, by De Giorgi's structure theorem \cite[Theorem 15.9]{Mag12}, the reduced boundary $\partial^\ast E$ and the measure-theoretic outer unit normal $\nu_E$ are such that, in $A$,
\eq{\label{defn:reduced-bdry-nu_E}
D\chi_E
=-\nu_E\mcH^n\llcorner\partial^\ast E,\quad
|D\chi_E|
=\mcH^n\llcorner\partial^\ast E.
}
Up to modification of sets of measure zero (see \cite[(15.3)]{Mag12}), we have and will always assume in this paper that
\eq{
\overline{\p^\ast E}
=\p E.
}

For every Borel set $G\subset A$, we define the \emph{anisotropic
perimeter} of $E$ in $G$ by
\eq{\label{defn:anisotropic-perimeter}
P_F(E;G)
\coloneqq
\int_{G\cap\partial^*E}F(\nu_E)\,\rd\mcH^n
=
\int_G F\left(-\frac{\rd D\chi_E}{\rd|D\chi_E|}\right)
\rd|D\chi_E|.
}

\begin{definition}\label{Defn:F-minimizer}
\normalfont
Let $A\subset\mbR^{n+1}$ be open. A set
$E\subset\mbR^{n+1}_+=\{x\in\mbR^{n+1}:x_1>0\}$ of locally finite perimeter in $A$ is called an
\emph{$F$-minimizer in $(A,\mbR^{n+1}_+)$} if
\eq{\label{defn:F-minimizer}
P_F(E;\mbR^{n+1}_+\cap W)\leq P_F(G;\mbR^{n+1}_+\cap W)
}
for every open set $W\subset\subset A$ and every set $G\subset \mbR^{n+1}_+$ of locally
finite perimeter in $A$ satisfying $E\De G\subset\subset W$.
\end{definition}

\subsection{Anisotropic free boundary minimal hypersurfaces}
We first collect some well-known facts.

Consider any $C^2$-hypersurface $M$ embedded in $\mbR^{n+1}_+$ with a chosen unit normal vector field, denoted by $\nu$.
Suppose that $M$ has non-empty boundary $\p M$, which lies entirely in $\p\mbR^{n+1}_+$, so that $M$ meets $\p\mbR^{n+1}_+$ transversally along $\p M$.
We denote by $\mu$  the outer unit co-normal to $\p M$ with respect to $M$, and then by $\bar\nu$ the outer unit co-normal to $\p M$ in $\p\mbR^{n+1}_+$, such that (noting that $-e_1$ is the outer unit normal to $\p\mbR^{n+1}_+$) the bases $\{\nu,\mu\}$ and $\{\bar\nu,-e_1\}$ have the same orientation in the normal bundle of $\p M$ (as a co-diminsion $2$ submanifold in $\mbR^{n+1}$).
Hence it is easy to see the following relations
\eq{\label{eq:mu-relation}
\mu=&-\left<\nu,-e_1\right>\bar\nu+\left<\mu,-e_1\right>(-e_1),\\
\nu=&\left<\mu,-e_1\right>\bar\nu+\left<\nu,-e_1\right>(-e_1).
}
Equivalently,
\eq{\label{eq:bar-nu-relation}
\bar\nu=&-\left<\nu,-e_1\right>\mu+\left<\mu,-e_1\right>\nu,\\
-e_1=&\left<\mu,-e_1\right>\mu+\left<\nu,-e_1\right>\nu.
}

Let $g$ denote the metric on $M$ induced from embedding in $\mbR^{n+1}$.
Denote by $\na$ the  Levi-Civita connection and  by ${\rm div}_M$ the  divergence operator on  $(M,g)$.
We use $h$ to denote the second fundamental form of $(M,g)\subset\mbR^{n+1}$, i.e.,
\eq{
h(X,Y)
=\left<\bar D_X\nu,Y\right>,\quad\forall X,Y\in TM.
}

Let $\nu_F=\bar DF(\nu)$ be the \textit{anisotropic normal} of $M$ and 
$A_F=\bar D^2F(\nu)$, where $\nu$ is the upwards pointing unit normal along $M$.
By \eqref{eq:Hessian-F}, $A_F$ can be identified as a positive definite symmetric $2$-tensor on $M$.
Hence, we define $h_F\coloneqq A_F\circ h$ as the \textit{anisotropic second fundamental form} on $M$; and $H_F\coloneqq{\rm tr}_g(h_F)$ as the \textit{anisotropic mean curvature}.
Note also that $h_F$ satisfies
\eq{\label{defn:h_F}
h_F(X,Y)
=\left<\bar D_X\nu_F,Y\right>
=(A_F\circ h)(X,Y),\quad\forall X,Y\in TM.
}

Let $A^{-1}_F$ denote the inverse of $A_F$, the \textit{$F$-anisotropic metric} on $M$ is then defined as
\eq{
g_F(X,Y)
=g(A^{-1}_FX,Y),\quad X,Y\in TM.
}
The \textit{$F$-anisotropic gradient} of a function $\phi$ defined on $M$ is given by
\eq{
\na_F\phi
=A_F\nabla \phi
=\bar D^2F(\nu)\na\phi,
}
and the \textit{$F$-Laplacian} is then given by
\eq{
\De_F\phi
={\rm div}_M(\na_F\phi)
={\rm div}_M\left(\bar D^2F(\nu)\na\phi\right).
}

\begin{definition}\label{Defn:anisotropic-free-bdry}
\normalfont
We call $M$ an anisotropic minimal hypersurface if $H_F=0$ on $M$.
We say that $M$ has free boundary in the anisotropic sense in $\mbR^{n+1}_+$, if $\left<\nu_F,e_1\right>=0$ on $\p M\subset\p\mbR^{n+1}_+$.
\end{definition}

As an immediate consequence of Lemma \ref{lem:normal controlled by tangential}, we have:
\begin{corollary}
Let $u$ be a $C^2$-function on $\mbR^n_+$, whose graph $M$ has free boundary in the anisotropic sense in $\mbR^{n+1}_+$.
Then the outward unit conormal $\mu$ of $\p M$ satisfies
\eq{\label{eq:mu-transverse}
 \langle\mu,-e_1\rangle
 =\frac{\sqrt{1+\abs{D'u}^2}}{\sqrt{1+\abs{Du}^2}}
 \ge c(F)>0.
}
\end{corollary}


\begin{proof}
As computed in \cite[Proof of Lemma A.1]{WW26}, we can write
\eq{
 \mu
 =\frac{\sqrt{1+\abs{Du}^2}}{\sqrt{1+|D'u|^2}}
 \left(-e_1-\frac{u_1}{\sqrt{1+\abs{Du}^2}}\nu\right),
}
by Lemma \ref{lem:normal controlled by tangential} and the fact that $f_1(Du)=0$ along the boundary, we thus obtain
\eq{
 \langle\mu,-e_1\rangle
 =\frac{\sqrt{1+|D'u|^2}}{\sqrt{1+\abs{Du}^2}}
 =\left(1+\frac{|u_1|^2}{1+|D'u|^2}\right)^{-1/2}
 \geq\left(1+C(F)^2\right)^{-1/2}.
}

\end{proof}

\begin{lemma}[{\cite[Proposition 4.1]{GX25}}]\label{Lem:principal-direction-mu}
Let $M$ be a hypersurface having free boundary in the anisotropic sense in $\mbR^{n+1}_+$, then the outer unit co-normal $\mu$ is an anisotropic principal direction of $\p M$ in $M$, i.e., $h_F(\tau,\mu)=0$ for any $\tau\in T(\p M)$. 
\end{lemma}
\begin{proof}
Using \eqref{eq:bar-nu-relation} we compute
\eq{
0
=\na_\tau\left<\nu_F,e_1\right>
=\left<\bar D_\tau\nu_F,e_1\right>
=&\left<\bar D_\tau\nu_F,\left<\mu,e_1\right>\mu+\left<\nu,e_1\right>\nu\right>\\
=&\left<\mu,e_1\right>\left<\bar D_\tau\nu_F,\mu\right>+\left<\nu,e_1\right>\left<\bar D_\tau\nu_F,\nu\right>
=\left<\mu,e_1\right>h_F\left(\tau,\mu\right).
}
Since $M$ meets $\p\mbR^{n+1}_+$ transversely, we have $\left<\mu,e_1\right>\neq0$ on $\p M$, the assertion then follows.
\end{proof}

However, on some occasions it is not easy to use $\mu$, but more convenient to use the following modified one.

\begin{definition}\label{defn:anisotropic-free-bdry-hypersurface-version}
\normalfont
Let $M$ be an embedded hypersurface in $\mbR^{n+1}$ with non-empty boundary $\p M$,
define the \textit{anisotropic co-normal} along $\p M$ as
\eq{\label{defn:mu_F}
\mu_F
\coloneqq\left<\nu_F,\nu\right>\mu-\left<\nu_F,\mu\right>\nu,
}
which clearly lies in the normal bundle of $\p M$.
\end{definition}
It was observed recently in \cite{GX25} that by the definitions of $\mu_F,\nu_F$, and the relation \eqref{eq:bar-nu-relation},  free boundary in the anisotropic sense can be characterized in terms of $\mu_F$ as follows:
\begin{lemma}[{\cite[Proposition 3.2]{GX25}}]
Let $M$ be an embedded hypersurface in $\mbR^{n+1}_+$ with non-empty boundary $\p M$ lying in $\p\mbR^{n+1}_+$, such that $M$ meets $\p\mbR^{n+1}_+$ transversally.
Then along $\p M$, one has
\eq{\label{eq:mu_F-relation}
\left<\mu_F,\bar\nu\right>
=&-\left<\nu_F,-e_1\right>,\\
\left<\mu_F,-e_1\right>
=&\left<\nu_F,\bar\nu\right>.
}
\end{lemma}
From its definition, it is clear that $\mu_F$ always lies in the normal space of $\p M$, which is spanned by $\mu$ and $\nu$, while in general the anisotropic normal $\nu_F$ does not. Thus it is more convenient to use $\mu_F$ instead of $\nu_F$ in the anisotropic setting.
\begin{lemma}\label{Lem:mu_F-perpendicular}
A hypersurface has free boundary in the anisotropic sense in $\mbR^{n+1}_+$ if and only if 
\eq{
\mu_F\perp\p\mbR^{n+1}_+.
}
\end{lemma}
\begin{proof}
It follows 
from \eqref{eq:mu_F-relation} and the free boundary condition \eqref{eq:<nu_F,e_1>=0} that 
$\left<\mu_F,\bar\nu\right>=0$. On the other hand, $\mu_F$ lies in the normal space of $\p M$. The assertion follows from the above two facts.
\end{proof}

To prove our main result, we need the following differential equation on anisotropic minimal hypersurface, in the spirit of \cite[Lemma 2.11]{WWXZ26}.

\begin{lemma}\label{Lem:V_F-PDE}
Let $M$ be a hypersurface in $\mbR^{n+1}_+$ with vanishing anisotropic mean curvature, let
$\nu_{n+1}(x)\coloneqq\left<\nu(x),e_{n+1}\right>$ where $e_{n+1}=(0,\cdots,0,1)$, and $V_F(x)\coloneqq\frac{\nu_{n+1}(x)}{F(\nu(x))}$, then
\eq{\label{eq:De_F-V_F}
\Delta_F V_F
=-2 g\left(\nabla \log F(\nu), \nabla_F\left(V_F\right)\right)-\frac{V_F}{F(\nu)} \operatorname{tr}_g\left(h_F^2\right).
}
\end{lemma}
\begin{proof}
By direct computations we obtain
\eq{
\Delta_F V_F
=&-\frac{2 g\left(\nabla \nu_{n+1}, \nabla_F(F(\nu))\right)}{F(\nu)^2}+\frac{2 \nu_{n+1}}{F(\nu)^3} g\left(\nabla(F(\nu)), \nabla_F(F(\nu))\right)-\frac{\nu_{n+1}}{F(\nu)^2} \operatorname{tr}_g\left(h_F^2\right)\\
=&-2 g\left(\frac{\nabla F}{F}, \frac{1}{F} \nabla_F \nu_{n+1}-\frac{\nu_{n+1}}{F^2} \nabla_F F\right)-\frac{\nu_{n+1}}{F(\nu)^2} \operatorname{tr}_g\left(h_F^2\right)\\
=&-2 g\left(\nabla \log F(\nu), \nabla_F\left(V_F\right)\right)-\frac{V_F}{F(\nu)} \operatorname{tr}_g\left(h_F^2\right),
}
where in the first equality we have used 
$$\Delta_F \nu_{n+1}+{\rm tr}_g(A_Fh^2)\nu_{n+1}=0,$$
and 
$$\Delta_F F+{\rm tr}_g(A_Fh^2)F={\rm tr}_g(h_F^2).$$
For the proof see  for instance \cite[Theorem 3.1]{Winklmann05-ArchMath} 
and  \cite[Lemma 2.8]{WWXZ26}.
\end{proof}


Similar to \cite[Lemma 2.12]{WWXZ26} we prove the following tangential property of $V_F$ along the boundary.
\begin{lemma}\label{Lem:tangential-property-bdry}
Let $M$ be a hypersurface with free boundary in the anisotropic sense in $\mbR^{n+1}_+$, let
$\nu_{n+1}(x)\coloneqq\left<\nu(x),e_{n+1}\right>$, and $V_F(x)\coloneqq\frac{\nu_{n+1}(x)}{F(\nu(x))}$, then
\eq{
g\left(\nabla_F V_F, \mu\right)\equiv0\text{ along the free boundary of }M.
}
\end{lemma}
\begin{proof}
Using \eqref{eq:<DF(z),z>=F(z)} we rewrite $V_F=\left<\nu_F,\nu\right>^{-1}\left<\nu,e_{n+1}\right>$.
By direct computations we obtain
\eq{
g(\na_FV_F,\mu)
=&\left<A_F[\na(V_F)],\mu\right>
=\left<A_F\left[\na\left(\left<\nu_F,\nu\right>^{-1}\left<\nu,e_{n+1}\right>\right)\right],\mu\right>\\
=&0-\left<\nu_F,\nu\right>^{-2}h_F\left(\nu^T_F,\mu\right)\left<\nu,e_{n+1}\right>+\left<\nu_F,\nu\right>^{-1}h_F(e^T_{n+1},\mu)\\
=&-\left<\nu_F,\nu\right>^{-2}h_F\left(\mu,\mu\right)\left(\left<\nu,e_{n+1}\right>\left<\nu_F,\mu\right>-\left<\nu_F,\nu\right>\left<\mu,e_{n+1}\right>\right)\\
=&\left<\nu_F,\nu\right>^{-2}h_F(\mu,\mu)\left<
{\left<\nu_F,\mu\right>\nu-\left<\nu_F,\nu\right>\mu},e_{n+1}\right>
\\
=&-F(\nu)^{-2}h_F(\mu,\mu) \langle{\mu_F},e_{n+1}\rangle=0.
}
For the third equality we have used Lemma \ref{Lem:principal-direction-mu},
while for the last one we have used 
Lemma \ref{Lem:mu_F-perpendicular}.
\end{proof}

\section{The mixed boundary problem}\label{Sec:3}
Fix $\varrho>0$, we continue to use the notations \eqref{defn:conventions-Om} with $\rho=\varrho$ in this Section, and consider the following mixed boundary problem:
\eq{
 \begin{cases}
 {\rm div}(Df(Dw))=0 & \text{in }\Omega_\varrho,\\
 \left< Df(Dw),e_1\right>=0 & \text{on }\Gamma_\varrho,\\
 \lim_{\Omega\ni x\to y}w(x)=\phi(y)
 &  \text{for every }y\in S_\varrho\setminus E_\varrho,
 \end{cases}
}
where $\phi$ is a $C^2$ function defined in a neighborhood of $\overline S_\varrho$.

Before we begin to prove Theorem \ref{thm:replacement}, we recall that Druet in \cite{druet2012classical} proved
the classical solvability of the contact angle problem for anisotropic equations of mean curvature type in a bounded domain by establishing a global gradient estimate. But this global gradient estimate is not sufficient for our purpose.

The following simple observation shows that the approximation \eqref{eq:uniform elliptic-PDE} is indeed uniformly elliptic.
\begin{lemma}\label{lem:Gq-coercive}
There exist $c(F),C(F)>0$ such that for all $p,\xi\in\mbR^n$, 
\eq{\label{eq:Gq-coercive}
 \abs{\xi}^2-\frac{\left< p,\xi\right>^2}{1+\abs{p}^2}
 +\left<Df(p),\xi\right>^2
 \geq c(F)\abs{\xi}^2
}
and 
\eq{\label{eq: uniformelliptic}
 c(F)\tau|\xi|^2
 \le f_{\tau,ij}(p)\xi_i\xi_j
 \le C(F)|\xi|^2.
}
\end{lemma}
\begin{proof}
For any $p,\xi\in\mbR^n$, let $z=(-p,1)$ and $Y=(-\xi,0)$.
Since $\bar D^2(F^2)$ is zero-homogeneous and positive definite on $\mbS^n$, we have
\eq{
\bar D^2(F^2)(z)[Y,Y]
\geq c(F)\abs{Y}^2
=c(F)\abs{\xi}^2,
}
where $\bar D^2(F^2)=2F\bar D^2F+2\bar DF\otimes\bar DF$, and hence the LHS can be rewritten as
\eq{
\frac12\bar D^2(F^2)(z)[Y,Y]
=f(p)f_{ij}(p)\xi_i\xi_j+\left<Df(p),\xi\right>^2.
}
By \eqref{eq:f-growth} and \eqref{eq:D2f-global}, we thus deduce
\eq{
c(F)\abs{\xi}^2
\leq f(p)f_{ij}(p)\xi_i\xi_j+\left<Df(p),\xi\right>^2
\leq C(F)\left(\abs{\xi}^2-\frac{\left< p,\xi\right>^2}{1+\abs{p}^2}
 +\left<Df(p),\xi\right>^2\right),
}
which implies \eqref{eq:Gq-coercive}.

By direct computation,
\eq{\label{eq:D2-ftau}
 D^2f_\tau
 =(1+\tau f)D^2f+\tau Df\otimes Df.
}
Using Lemma \ref{lem:structure of f} and \eqref{eq:Gq-coercive}, we find
\eq{
D^2f_\tau(p)[\xi,\xi]
\geq c(F)\tau\left(\abs{\xi}^2-\frac{\left< p,\xi\right>^2}{1+\abs{p}^2}
 +\left<Df(p),\xi\right>^2\right)
\geq c(F)\tau\abs{\xi}^2,
}
which proves the lower bound in \eqref{eq: uniformelliptic}.
The upper bound follows similarly, and the proof is thus completed.
\end{proof}

For $0<\tau\leq1$ fixed, we now minimize \eqref{eq:Atau}.

\begin{proposition}\label{prop:mixed boundary problem}
For every $\tau\in(0,1]$ there is a unique minimizer
\eq{
 w_\tau\in W^{1,2}(\Omega_\varrho),
 \quad w_\tau |_{{S_\varrho}}
 =\phi,
}
of \eqref{eq:Atau}.
Moreover, the following properties hold:
\begin{enumerate}
    \item For every $\zeta\in W^{1,2}(\Omega_\varrho)$ with $\zeta=0$ on $S_\varrho$,
\eq{\label{eq:weak-ftau}
\int_{\Om_\varrho}\left<Df_\tau(Dw_\tau),D\zeta\right>\rd x=0.
}

\item There exists $C>0$ independent of the choice of $\tau$, such that \eq{\label{eq:energy-uniform}
 \int_{\Om_\varrho}\abs{Dw_\tau}\rd x\leq C,
 \text{ and }\quad 
 \tau\int_{\Omega_\varrho} f(Dw_\tau)^2\rd x\leq C.
}

\item For any $r\in(0,\varrho)$, there exists  $\gamma=\gamma(n,F,\tau)\in(0,1)$ such that $w_\tau\in C^{2, \gamma}(B^{n+}_{r}\cup\Gamma_{r})$ and satisfies
\eq{\label{eq:c2 solution to local equation}
 \begin{cases}
 {\rm div}(Df_\tau(Dw_\tau))=0 & \text{in }B^{n+}_{r},\\[1mm]
 \left<Df_\tau(Dw_\tau),e_1\right>=0 & \text{on }\Gamma_{r}.
 \end{cases}
}
\end{enumerate}
\end{proposition}

\begin{proof}
Let $\tilde\phi\in W^{1,2}(\Om_\varrho)$ be an extension such that its trace on $S_\varrho$ is $\phi$, and let $V_0=\left\{\zeta\in W^{1,2}(\Om_\varrho):\zeta\mid_{S_\varrho}=0\right\}$.
By \eqref{eq:f-growth}, for each fixed $\tau>0$ we have
\eq{
\mathscr A_\tau(v)
\geq c(F)\tau\int_{\Om_\varrho}\abs{Dv}^2\rd x,\quad\forall v\in\tilde\phi+V_0.
}
Applying Poincar\'e inequality for $v-\tilde\phi\in V_0$, we have $\norm{v-\tilde\phi}_{L^2(\Om_\varrho)}\leq C(\Om_\varrho)\norm{Dv-D\tilde\phi}_{L^2(\Om_\varrho)}$, it follows that
\eq{
\norm{v}^2_{W^{1,2}(\Om_\varrho)}
\leq C(F,\Om_\varrho)\left(\norm{\tilde\phi}^2_{W^{1,2}(\Om_\varrho)}+\frac1\tau\mathscr A_\tau(v)\right),
}
implying that the functional $\mathscr A_\tau(v)$ is coercive. The existence of the minimizer then follows from the direct method. Moreover, $f_\tau$ is strictly convex thanks to \eqref{eq: uniformelliptic}, which shows that the minimizer is unique.

To proceed, we note that \eqref{eq:weak-ftau} follows by computing the first variation.
Comparing the minimizer $w_\tau$ with $\tilde\phi$, we find 
\eq{
\mathscr A_\tau(w_\tau)
=\int_{\Om_\varrho} f(Dw_\tau)\rd x+\frac\tau2\int_{\Om_\varrho} f^2(Dw_\tau)\rd x
\leq\int_{\Om_\varrho} f(D\tilde\phi)\rd x+\frac12\int_{\Om_\varrho} f^2(D\tilde\phi)\rd x
\eqqcolon C,
}
where we have used $\tau\leq1$.
Since $f(p)\geq c(F)\abs{p}$, this proves \eqref{eq:energy-uniform}. 

Finally, higher regularity of $w_\tau$ follows from the standard regularity argument for uniformly elliptic equations in divergence form with conormal boundary conditions,
cf. \cite{LT86,Lieberman88,Lieberman13}. Here we only point out that $\gamma$ is independent of $r\in(0,\varrho)$ but depends on $\tau$ due to \eqref{eq: uniformelliptic}. This completes the proof.
\end{proof}

In the spirit of \cite{Simon76}, to let $\tau\searrow0$ we need to establish the following gradient estimate for the approximated equation \eqref{eq:uniform elliptic-PDE}.
{
\begin{theorem}\label{uniform gradient estimate}
There exists a positive constant $C(n,F)$, depending only on $n$ and $F$ while independent of $\tau$, with the following properties.

If $x_0\in\Gamma_\varrho$ and $R>0$ satisfy $B^{n+}_{2R}(x_0)\subset\Om_\varrho$, also
\eq{
\Gamma_{2R}(x_0)
\coloneqq B^n_{2R}(x_0)\cap\p\mbR^n_+
\subset\Gamma_\varrho,
}
then there holds
\eq{\label{eq:uniform-gradient}
 \sup_{B^{n+}_{R/2}(x_0)\cup\Gamma_{R/2}(x_0)}|Dw_\tau|
 \le C(n,F)\exp\left[
C(n,F)\left(1+\frac{{\rm osc}_{B^{n+}_{2R}(x_0)}w_\tau}{R}\right)^{n+3}
 \right].
}
If $x_0\in\Om_\varrho$ and $B^n_{2R}(x_0)\subset\Om_\varrho$, then
\eq{\label{eq:uniform-gradient-interior-ball}
\sup_{B^n_{R/2}(x_0)}\abs{Dw_\tau}
\leq C(n,F)\exp\left[
C(n,F)\left(1+\frac{{\rm osc}_{B^{n}_{2R}(x_0)}w_\tau}{R}\right)^{n+3}
 \right].
}
\end{theorem}
}
Since $w_\tau$ satisfies the anisotropic free boundary condition, to establish Theorem \ref{uniform gradient estimate}, we will follow the strategy of our recent gradient estimate in \cite{WWXZ26}. 
We note that it suffices to consider the case $x_0=0$, since the estimate for $x_0\in\Gamma_\varrho$ would follow after tangential translation along $\p\mbR^n_+$, while for $x_0\in\Om_\varrho$ with $B^n_{2R}(x_0)\subset\Om_\varrho$, the proof is identical and in fact simpler, because there is no boundary term appearing in the integral method proof.
However, we note that an essential difference to \cite{WWXZ26} is that the approximated equation is not an AMSE, therefore we cannot directly use the intrinsic PDE \eqref{eq:De_F-V_F} on the anisotropic minimal graph as in \cite{WWXZ26}, but need to seek for a similar one.

In the rest of this Section, let $u$ be a solution of \eqref{eq:c2 solution to local equation} with $\tau$ fixed.
We continue to use the following notations of standard graphical area element and anisotropic graphical area element in our previous work \cite{WWXZ26}:
\eq{
W=\sqrt{1+\abs{Du}^2},\quad
W_f=f(Du).
}
and we use for simplicity the conventions
\eq{\label{eq:Atau-coefficient}
 \mcA_\tau^{ij}
 =W_f^2 f_{\tau,ij}(Du),
 \quad
 z=1+\log W_f,\quad
 g^{ij}
 =\de_{ij}-\frac{u_iu_j}{W^2},\quad
 h_{ij}
 =\frac{u_{ij}}W.
}
As in \cite[Remark 3.1]{WWXZ26},
we could assume WLOG that $W_f\geq1$, and consequently $z\geq1$.

\begin{proposition}\label{prop:entropy}

There are $c(F),C(F)>0$ independent of $\tau$, such that for every
$\xi\in\mbR^n$,
\eq{\label{eq:Btau-bounds}
 c(F)Wg^{ij}\xi_i\xi_j
 \le\frac{\mathcal A_\tau^{ij}}{1+\tau W_f}\xi_i\xi_j
 \le C(F)W|\xi|^2.
}
In the distributional sense, we have
\eq{\label{eq:entropy-identity}
 D_i(\mathcal A_\tau^{ij}z_j)
 =\mathcal A_\tau^{ij}z_i z_j
 +W_f f_{\tau,ij}(Du)f_{k\ell}(Du)u_{\ell i}u_{kj},
}
and also
\eq{\label{eq:Btau-curvature}
 D_i\left(\frac{\mathcal A_\tau^{ij}}{1+\tau W_f}z_j\right)
 \ge c(F)W|h|^2.
}
If $\langle Df(Du),e_1\rangle=0$ in $\Gamma_\varrho$, then
\eq{\label{eq:free boundary}
 \mathcal A_\tau^{1j}z_j=0
 \quad\text{on }\Gamma_\varrho.
}
\end{proposition}


\begin{proof}
By definition,
\eq{\label{eq:A^ij_tau-1+tau-W_f}
 \frac{\mathcal A_\tau^{ij}}{1+\tau W_f}\xi_i\xi_j
 =W_f^2f_{ij}(Du)\xi_i\xi_j
 +\frac{\tau W_f^2}{1+\tau W_f}f_i(Du)f_j(Du)\xi_i\xi_j.
}
The first term is bounded from below by $c(F)Wg^{ij}$ thanks to \eqref{eq:D2f-global} and $W_f\geq c(F)W$, while the second term is non-negative. For the upper bound in \eqref{eq:Btau-bounds}, note that both terms on the RHS of \eqref{eq:A^ij_tau-1+tau-W_f} are bounded above by $C(F)W\abs{\xi}^2$, thanks to $|Df|\le C(F)$ and
\eq{
 \frac{\tau W_f^2}{1+\tau W_f}\le W_f\le C(F)W,
}
which proves \eqref{eq:Btau-bounds}.

Since $D_i(W_f)=f_k(Du)u_{ki}$ and $z_i=\frac{D_i(W_f)}{W_f}$, differentiating ${\rm div}(Df_\tau(Du))=0$ with respect to $x_k$ gives, in the distributional sense,
\eq{
D_i\left(f_{\tau,ij}(Du)u_{jk}\right)
=0.
}
Hence it is direct to compute
\eq{
D_i(\mcA^{ij}_\tau z_j)
=D_i(W_ff_{\tau,ij}f_ku_{kj})
=&D_i(W_f)f_{\tau,ij}D_j(W_f)
  +W_ff_{\tau,ij}f_{k\ell}u_{\ell i}u_{kj}
  +W_ff_k\underbrace{D_i(f_{\tau,ij}u_{kj})}_{=0},
}
the first term on the right can be rewritten as $\mcA^{ij}_\tau z_iz_j$, which proves \eqref{eq:entropy-identity}.

By $D_iW_f=W_fz_i$ and a direct computation using \eqref{eq:entropy-identity}, we get
\eq{\label{eq:Btau-exact}
 D_i\left(\frac{\mathcal A_\tau^{ij}}{1+\tau W_f}z_j\right)
 &=\underbrace{\frac{\mathcal A_\tau^{ij}z_i z_j}{(1+\tau W_f)^2}}_{\geq0}+
\frac{W_f}{1+\tau W_f}
 f_{\tau,ij}(Du)f_{k\ell}(Du)u_{\ell i}u_{kj}.
}
Since $f_{\tau,ij}=(1+\tau f)f_{ij}+\tau f_i f_j$ and $D^2f$ is positive definite, we get
\eq{
 f_{\tau,ij}f_{k\ell}u_{\ell i}u_{kj}
 \geq(1+\tau W_f)f_{ij}f_{k\ell}u_{\ell i}u_{kj}+0
 \geq c(F)(1+\tau W_f)|h|^2,
}
where the last inequality follows from \eqref{eq:D2f-global} and
\eq{\label{eq:h-graph}
 |h|^2=g^{ik}g^{j\ell}h_{ij}h_{k\ell}
 =\frac1{W^2}g^{ik}g^{j\ell}u_{ij}u_{k\ell}.
}
Back to \eqref{eq:Btau-exact}, we thus deduce \eqref{eq:Btau-curvature}.

Finally, on $\Gamma_\varrho$ we have $f_1(Du)=0$, and hence $ \langle Df_{\tau}(Du), e_1\rangle=0$.
Taking tangential differentiation, we get
\eq{
f_{\tau,1j}(Du)u_{j\alpha}=0,\quad\text{for any }\alpha=2,\ldots,n.
}
Hence \eqref{eq:free boundary} follows from the fact that
\eq{
\mcA_\tau^{1j}z_j
=W_ff_{\tau,1j}(Du)f_k(Du)u_{kj}
=W_f\underbrace{f_1}_{=0}f_{\tau,1j}u_{1j}+W_f\sum_{\alpha=2}^nf_\alpha\underbrace{f_{\tau,1j}u_{j\alpha}}_{=0}
=0.
}
This completes the proof.
\end{proof}

\subsection{Integral estimates}
Following \cite[Section 3]{WWXZ26} but with a slight difference, we show:

\begin{lemma}\label{prop:upper bound of log integral}
For $R\leq\frac{1}2\varrho$ and any integer $m\geq0$, one has
\eq{\label{eq:moment-final}
 \int_{B^{n+}_{R}}W_{f,\tau}z^m\rd x
 \leq C(m,n,F)R^n\left(1+\frac{{\rm osc}_{B^{n+}_{2R}}u}R\right)^{m+2},
}
where $W_{f,\tau}\coloneqq W_f(1+\tau W_f)$, and $C(m,n,F)>0$ is independent of $\tau$.  
\end{lemma}

\begin{proof}
For $r<s<2R$, let $0\leq\eta\leq1$ be a standard cutoff function satisfying
\eq{
\eta
\equiv1\text{ on }B^{n+}_{r},\quad
{\rm spt}\eta\subset B^n_s,\quad
\abs{D\eta}
\leq\frac{2}{s-r}.
}
For each $k\geq0$, we define
\eq{
\mathscr E_k(t)
\coloneqq\int_{B^{n+}_t}W_{f,\tau}z^k\rd x.
}

\noindent{\bf Step 1. We prove that for any $m\geq1$, 
\eq{\label{standard formula 1}
\int_{B^{n+}_{s}}\eta^2z^{m-1}
\mcA_\tau^{ij}z_i z_j\rd x
\leq\frac{C(F)}{(s-r)^2}
\mathscr E_{m-1}(s).
}
}

Note that
\eq{
f_{\tau,ij}f_{k\ell}u_{\ell i}u_{kj}
={\rm tr}\left(D^2f_\tau D^2u D^2f D^2u\right)
=\norm{(D^2f_\tau)^{1/2}D^2u(D^2f)^{1/2}}^2
\geq0.
}
Testing \eqref{eq:entropy-identity} with $\eta^2z^{m-1}$ over $B^{n+}_{s}$, using integration by parts and taking into account the boundary relation \eqref{eq:free boundary} yields
\eq{
 &\int_{B^{n+}_{s}}\eta^2z^{m-1}
   \mathcal A_\tau^{ij}z_i z_j\rd x
 +(m-1)\int_{B^{n+}_{s}}\eta^2z^{m-2}
   \underbrace{\mathcal A_\tau^{ij}z_i z_j}_{\geq0}\rd x\\
&\leq -2\int_{B^{n+}_{s}}\eta z^{m-1}
   \mathcal A_\tau^{ij}\eta_i z_j\rd x\\
&\leq \int_{B^{n+}_{s}}\frac12\eta^2z^{m-1}
       \mathcal A_\tau^{ij}z_i z_j\rd x
 +2\int_{B^{n+}_{s}}z^{m-1}\mathcal A_\tau^{ij}\eta_i\eta_j\rd x.
}
Consequently,
\eq{
 \int_{B^{n+}_{s}}\eta^2z^{m-1}
 \mathcal A_\tau^{ij}z_i z_j\rd x
 &\leq4\int_{B^{n+}_{s}}z^{m-1}
 \mathcal A_\tau^{ij}\eta_i\eta_j\rd x\\
 &\overset{\eqref{eq:Btau-bounds}}{\leq}
 C(F)\int_{B^{n+}_{s}}
 W_{f,\tau}z^{m-1}|D\eta|^2\rd x\\
 &\leq\frac{C(F)}{(s-r)^2}
 \int_{B^{n+}_{s}}
 W_{f,\tau}z^{m-1}\rd x
 =\frac{C(F)}{(s-r)^2}\mathscr E_{m-1}(s),
}
which completes {\bf Step 1}.

\smallskip
\noindent{\bf Step 2. For $m\in \mathbb{N}$, we have
\eq{\label{eq:iteration}
 \mathscr E_m(r)
 &\le\frac12\mathscr E_m(s)
 +C(m,n,F)R^n
 \left(1+\frac{{\rm osc}_{B^{n+}_{2R}}u}{s-r}\right)^{m+2}+C(m,n,F)\frac{{\rm osc}_{B^{n+}_{2R}}u}{s-r}\mathscr E_{m-1}(s).
}
}

We first assume $m\geq1$. Replacing $u$ by $u-\inf_{B_{2R,+}}u$, we have $0\leq u\leq{\rm osc}_{B^{n+}_{2R}}u$.
Multiplying \eqref{eq:c2 solution to local equation} with  $u\eta^2z^m$,
\eq{\label{eq:test-w}
 &\int_{B^{n+}_{s}}\eta^2z^m(1+\tau W_f)
 \langle Df(Du),Du\rangle\rd x\\
 &\quad=-2\int_{B^{n+}_{s}}u\eta z^m(1+\tau W_f)
 \langle Df(Du),D\eta\rangle\rd x\\
 &\qquad-m\int_{B^{n+}_{s}}u\eta^2z^{m-1}(1+\tau W_f)
 \langle Df(Du),Dz\rangle\rd x.
}

Since $\left<Df(p),p\right>=f(p)-\bar D_{n+1}F(-p,1)$ and $\bar D_{n+1}F(-p,1)\leq C(F)$, it is easy to see
\eq{\label{eq:coercive-term}
 (1+\tau W_f)\left<Df(Du),Du\right>z^m
 \geq c(F)W_{f,\tau}z^m-C(m,F).
}
By direct computation, we also find that for any fixed $\delta>0$, if $W_f\geq\frac{C(F)}\de\frac{{\rm osc}_{B^{n+}_{2R}}u}{s-r}$ then
\eq{
C(F)\frac{{\rm osc}_{B^{n+}_{2R}}u}{s-r}(1+\tau W_f)z^m
\leq\de W_{f,\tau}z^m.
}
In the other case, we have $W_f\leq C(F,\de)\left(1+\frac{{\rm osc}_{B^{n+}_{2R}}u}{s-r}\right)$, and since $z=1+\log W_f\leq W_f$ (thanks to $W_f\geq1$), we find using $0<\tau\leq1$ that
\eq{
\frac{{\rm osc}_{B^{n+}_{2R}}u}{s-r}(1+\tau W_f)z^m
\leq C(m,F,\de)\left(1+\frac{{\rm osc}_{B^{n+}_{2R}}u}{s-r}\right)^{m+2}.
}
From these two cases we get for any fixed $\de>0$,
\eq{\label{eq:pointwise bound}
 C(F)\frac{{\rm osc}_{B^{n+}_{2R}}u}{s-r}(1+\tau W_f)z^m
 \le \delta W_{f,\tau}z^m
 +C(m,F,\delta)\left(1+\frac{{\rm osc}_{B^{n+}_{2R}}u}{s-r}\right)^{m+2}.
}
Combining, we obtain
\eq{\label{eq:first-error}
 &\Abs{2\int_{B^{n+}_{s}}u\eta z^m(1+\tau W_f)
 \langle Df(Du),D\eta\rangle\rd x}\\
 &\le C(F)\frac{{\rm osc}_{B^{n+}_{2R}}u}{s-r}
 \int_{B_{^{n+}s}}(1+\tau W_f)z^m\rd x\\
 &\overset{\eqref{eq:pointwise bound}}{\le}
 \de\int_{B^{n+}_{s}}W_{f,\tau}z^m\rd x
 +C(m,n,F,\delta)|B^{n+}_{s}|
 \left(1+\frac{{\rm osc}_{B^{n+}_{2R}}u}{s-r}\right)^{m+2}\\
 &\le\de\mathscr E_m(s)
 +C(m,n,F,\delta)R^n
 \left(1+\frac{{\rm osc}_{B^{n+}_{2R}}u}{s-r}\right)^{m+2}.
}
where we have used $0\le u\leq{\rm osc}_{B^{n+}_{2R}}u$, $0\le\eta\le1$, $|Df|\le C(F)$ and 
$|D\eta|\le2/(s-r)$.

On the other hand, by \eqref{eq:D2f-global} and $\abs{Df}\leq C(F)$, one has
\eq{
\mathscr F\coloneqq Df(Du)^T(D^2f(Du))^{-1}Df(Du)
\leq C(F)W^3
\leq C(F)W_f^3,
}
and hence thanks to $W_f\geq1$, $0<\tau\leq1$,
\eq{\label{ineq:mathscr-F-tau}
(1+\tau W_f)^2\mathscr F
=\mathscr F+2\tau W_f\mathscr F+\tau^2W_f^2\mathscr F
\leq C(F)W_f^3\left(1+\tau W_f+\tau\mathscr F\right).
}
By the Sherman-Morrison formula, we thus find (noting that $(1+\tau W_f)^2\leq2+2\tau^2W_f^2$)
\eq{\label{ineq:quadratic-form-mcA_tau-inver}
&((1+\tau W_f)Df)^T\mathcal A_\tau^{-1}
((1+\tau W_f)Df)
=\frac{(1+\tau W_f)^2Df^T(D^2f)^{-1}Df}
{W_f^2\left(1+\tau W_f+\tau Df^T(D^2f)^{-1}Df\right)}\\
\leq&\frac{2}{W_f^2}Df^T(D^2f)^{-1}Df+2\tau
\overset{\eqref{ineq:mathscr-F-tau}}{\leq} C(F)W_f.
}
We can now use Cauchy's inequality to get
\eq{
 &(1+\tau W_f)|\langle Df(Du),Dz\rangle|\\
 &\le
 \left(\mathcal A_\tau^{ij}z_i z_j\right)^{1/2}
 \left(
 ((1+\tau W_f)Df(Du))^T\mathcal A_\tau^{-1}
 ((1+\tau W_f)Df(Du))
 \right)^{1/2}\\*
 &\le C(F)W_f^{1/2}
 \left(\mathcal A_\tau^{ij}z_i z_j\right)^{1/2}.
}
For $m\ge1$ this gives
\eq{\label{eq:second-error}
 &m{\rm osc}_{B^{n+}_{2R}}u\int_{B^{n+}_{s}}\eta^2z^{m-1}(1+\tau W_f)
 |\langle Df(Du),Dz\rangle|\rd x\\
 \leq&C(m,F){\rm osc}_{B^{n+}_{2R}}u
 \left(\int_{B^{n+}_{s}}\eta^2z^{m-1}
 \mathcal A_\tau^{ij}z_i z_j\rd x\right)^{1/2}
 \left(\int_{B^{n+}_{s}}\eta^2z^{m-1}W_f\rd x\right)^{1/2}
 \\
 \overset{\eqref{standard formula 1}}{\le}&
 \frac{C(m,F){\rm osc}_{B^{n+}_{2R}}u}{s-r}\mathscr E_{m-1}(s)^{1/2}
 \left(\int_{B^{n+}_{s}}W_{f,\tau}z^{m-1}\rd x\right)^{1/2}\\
 =&C(m,F)\frac{{\rm osc}_{B^{n+}_{2R}}u}{s-r}\mathscr E_{m-1}(s).
}
Recalling that $\eta=1$ on $B^{n+}_{r}$, by \eqref{eq:coercive-term} we find
\eq{
 c(F)\mathscr E_m(r)
 &\le c(F)\int_{B^{n+}_{s}}\eta^2W_{f,\tau}z^m\rd x\\
 &\le\int_{B^{n+}_{s}}\eta^2z^m(1+\tau W_f)
 \langle Df(Du),Du\rangle\rd x+C(m,n,F)R^n.
}
Combining this with \eqref{eq:test-w}, \eqref{eq:first-error}, \eqref{eq:second-error}, then letting $\delta=c(F)/2$, we obtain \eqref{eq:iteration}.
For the case $m=0$, the $Dz$ term is zero, and the same calculation gives
\eq{\label{eq:radius-rec0}
 \mathscr E_0(r)
 \le\frac12\mathscr E_0(s)
 +C(F)R^n\left(1+\frac{{\rm osc}_{B^{n+}_{2R}}u}{s-r}\right)^2.
}
This finishes {\bf Step 2}.
Finally, by a standard iteration argument using \eqref{eq:iteration}, we deduce
\eq{
\mathscr E_m(r)
\leq C(m,n,F)R^n\left(1+\frac{{\rm osc}_{B^{n+}_{2R}}u}{\frac{3}2R-r}\right)^{m+2},\quad\forall r\in[R,\frac{3R}2).
}
Letting $r=R$ then yields the required estimate \eqref{eq:moment-final}.

\end{proof}

Here we point out that, an $L^1$-estimate ($m=1$) is sufficient for the gradient estimate in \cite{WWXZ26}, but here we need the estimate for large $m$ as the initial step for our iteration argument below.

\subsection{Moser-type iteration}

Now we use the strategy of \cite[Section 4]{WWXZ26}, and we divide the proof into several Lemmatum.

\begin{lemma}\label{lem:sobolev}
Let $\Sigma={\rm graph}(u)$. For any non-negative Lipschitz $\varphi$ vanishing on ${\rm graph}\left(u\mid_{S_\varrho}\right)$ and any $r>0$, one has
\eq{\label{eq:sobolev}
 \left(\int_\Sigma\varphi^{\frac{2n}{n-1}}\rd\mcH^n\right)^{\frac{n-1}{n}}
 \le C(n,F)\left[
 \frac1r\int_\Sigma\varphi^2\rd\mcH^n
 +r\int_\S\left(|\nabla\varphi|^2+|h|^2\varphi^2\right)\rd\mcH^n\right].
}
\end{lemma}
\begin{proof}
The proof is already contained in that of \cite[Proposition 4.3]{WWXZ26}, the only difference is that in proving \eqref{eq:sobolev}, one does not (and cannot) use the stability inequality \cite[(2.12)]{WWXZ26} since the graph is not anisotropic minimal.
\end{proof}

\begin{lemma}\label{lem:iteration}
For every $\beta\ge1$ and $0<r<s<\varrho$, put for simplicity $\chi=\frac n{n-1}$, and $\S_{t,+}\coloneqq{\rm graph}(u\mid_{B^{n+}_t})$ for any $0<t<\varrho$, then
\eq{\label{eq:one-step}
 \left(\int_{\S_{r,+}}z^{\beta\chi}\rd\mcH^n\right)^{1/\chi}
 \leq C(n,F)\frac{s}{(s-r)^2}
 \int_{\S_{s,+}}z^{\beta+1}\rd\mcH^n.
}

\end{lemma}

\begin{proof}
Choose as before the standard cutoff function $0\le\eta\le1$ with $\eta\equiv1$ on $B^{n+}_r$, ${\rm spt}\eta\subset B^n_s$, and $|D\eta|\le\frac{2}{s-r}$.

Multiplying both sides of \eqref{eq:Btau-curvature} by $\eta^2z^\beta$, integration by parts using \eqref{eq:free boundary} gives (note that $\rd\mcH^n=W\rd x$)
\eq{
 c(F)\int_{\S_{s,+}}\eta^2z^\beta|h|^2
 \rd\mcH^n
 &\leq\int_{B^{n+}_{s}}\eta^2z^\beta
 D_i\left(\frac{\mathcal A_\tau^{ij}}{1+\tau W_f}z_j\right)\rd x\\
 &=-\int_{B^{n+}_{s}}
 \frac{\mathcal A_\tau^{ij}}{1+\tau W_f}z_j
 D_i(\eta^2z^\beta)\rd x\\
 &=-\beta\int_{B^{n+}_{s}}\eta^2z^{\beta-1}
 \frac{\mathcal A_\tau^{ij}}{1+\tau W_f}z_i z_j\rd x-2\int_{B^{n+}_{s}}\eta z^\beta
 \frac{\mathcal A_\tau^{ij}}{1+\tau W_f}\eta_i z_j\rd x\\
 &\leq-\frac{\beta}{2}\int_{B^{n+}_{s}}\eta^2z^{\beta-1}
 \frac{\mathcal A_\tau^{ij}}{1+\tau W_f}z_i z_j\rd x
 +\frac{2}{\beta}\int_{B^{n+}_{s}}z^{\beta+1}
 \frac{\mathcal A_\tau^{ij}}{1+\tau W_f}\eta_i\eta_j \rd x.
}
Therefore, by the upper bound in \eqref{eq:Btau-bounds}, we deduce
\eq{\label{eq:Moser-energy}
&\frac\beta2\int_{B^{n+}_{s}}\eta^2z^{\beta-1}\frac{\mathcal A_\tau^{ij}}{1+\tau W_f}z_i z_j\rd x+c(F)\int_{\S_{s,+}}\eta^2z^\beta|h|^2\rd \mcH^n\\
\leq&\frac{C(F)}{\beta}\int_{\Sigma_{s,+}}z^{\beta+1}|D\eta|^2\rd\mcH^n.
}

To proceed, we take $\varphi=\eta z^{\beta/2}$ in \eqref{eq:sobolev} (with $r$ therein chosen as $s$) and get
\eq{\label{iteration 1}
&\left(
\int_{\Sigma_{s,+}}(\eta z^{\beta/2})^{2\chi}\,d\mathcal H^n
 \right)^{1/\chi}\\
\leq&C(n,F)\left[\frac1s\int_{\Sigma_{s,+}}\eta^2z^\beta
 \rd\mcH^n
 +s\int_{\Sigma_{s,+}}
 \left(|\nabla\varphi|^2+\eta^2z^\beta|h|^2\right)
 \rd\mcH^n\right].
}
We assert that
\eq{
\int_{\Sigma_{s,+}}
|\nabla\varphi|^2\rd\mcH^n
\leq C(F)\int_{\Sigma_{s,+}}z^{\beta+1}|D\eta|^2\rd\mcH^n.
}
To see this, note that 
$|\nabla\varphi|^2
=g^{ij}D_i\varphi D_j\varphi
\le2z^\beta g^{ij}\eta_i\eta_j
 +\frac{\beta^2}{2}\eta^2z^{\beta-2}g^{ij}z_i z_j$,
and hence
\eq{\label{eq:Moser-gradient-control} \int_{\Sigma_{s,+}}|\nabla\varphi|^2\rd\mcH^n
&\overset{\eqref{eq:Btau-bounds}}{\leq}2\int_{\Sigma_{s,+}}z^\beta|D\eta|^2
 \rd\mcH^n+C(F)\beta^2
 \int_{B^{n+}_{s}}\eta^2z^{\beta-2}
 \frac{\mathcal A_\tau^{ij}}{1+\tau W_f}z_i z_j\rd x\\
&\leq2\int_{\Sigma_{s,+}}z^{\beta+1}|D\eta|^2
 \rd\mcH^n+C(F)\beta^2
 \int_{B^{n+}_{s}}\eta^2z^{\beta-1}
 \frac{\mathcal A_\tau^{ij}}{1+\tau W_f}z_i z_j\rd x\\
 &\overset{\eqref{eq:Moser-energy}}{\le}
C(F)\int_{\Sigma_{s,+}}z^{\beta+1}|D\eta|^2
 \rd\mcH^n,
}
where we have used $z\geq1$ for the second inequality.
This proves the asserted estimate.
Combining with \eqref{eq:Moser-energy} and \eqref{iteration 1}, and using $0\leq\eta\leq1, z\geq1$, we deduce
\eq{
\left(\int_{\S_{r,+}}z^{\beta\chi}\rd\mcH^n\right)^{1/\chi}
\leq&\left(
 \int_{\Sigma_{s,+}}
 (\eta z^{\beta/2})^{2\chi}\rd\mcH^n
 \right)^{1/\chi}\\
\leq&C(n,F)\left(
\frac1s\int_{\Sigma_{s,+}}z^{\beta+1}
 \rd\mcH^n
+s\int_{\Sigma_{s,+}}z^{\beta+1}|D\eta|^2
 \rd\mcH^n\right)\\
\leq&C(n,F)\frac{s}{(s-r)^2}
 \int_{\Sigma_{s,+}}z^{\beta+1}\rd\mcH^n.
}
The proof is completed.
\end{proof}


\begin{lemma}\label{lem:shifted-Moser}
Assume $0<R<\varrho$ and that \eqref{eq:one-step} holds whenever $R/2<r<s\leq R$.
Then for any $p_0>n$,
\begin{equation}\label{eq:shifted-Moser}
 \|1+\log W_f\|_{L^\infty(\Sigma_{R/2,+})}
 \le C(n,F,p_0)
 \left(
 R^{-n}\int_{\Sigma_{R,+}}(1+\log W_f)^{p_0}\rd\mcH^n
 \right)^{1/(p_0-n)}.
\end{equation}
\end{lemma}

\begin{proof}
Let $p_{j+1}=\chi(p_j-1)$ and $R_j=\frac{R}2+\frac{R}{2^{j+1}}$, then
\eq{
p_j-n
=\chi^j(p_0-n),\quad
R_j\searrow\frac{R}2.
}
In particular, $p_j\to\infty$ and $\beta_j\coloneqq p_j-1\geq1$.

Applying \eqref{eq:one-step} with $\beta=\beta_j$, $s=R_j$, $r=R_{j+1}$, and putting $\mathcal Y_j\coloneqq\left(R^{-n}\int_{\S_{R_j,+}}z^{p_j}\rd\mcH^n\right)^{1/p_j}$, we get
\eq{
\mathcal Y_{j+1}
\leq\bigl(C(n,F)4^j\bigr)^{1/(p_j-1)}
 \mathcal Y_j^{p_j/(p_j-1)}.
}
Moreover, we note that
\eq{
\prod_{j=0}^{N}\frac{p_j}{p_j-1}
 =
 \frac{\chi^{N+1}p_0}{p_{N+1}}
 \to
 \frac{p_0}{p_0-n}.
}
The rest of the proof is then standard.

\end{proof}

\begin{remark}
\normalfont
Here $p_0>n$ is different from the one in \cite[Proposition 4.6]{WWXZ26}, which is because the graph is not an anisotropic minimal graph, so an extra term $\int |h|^2\varphi^2\,d\mathcal H^n$ appears in the boundary Sobolev inequality \eqref{eq:sobolev}, compared with \cite[Proposition 4.3]{WWXZ26}.
\end{remark}

\begin{proof}[Proof of Theorem \ref{uniform gradient estimate}]
Taking $p_0=n+1$ in Lemma \ref{lem:shifted-Moser}, we get
\eq{\label{eq:gradient-proof-Moser}
 \|z\|_{L^\infty(\Sigma_{R/2,+})}
 \leq&C(n,F)R^{-n}
 \int_{\Sigma_{R,+}}z^{n+1}\rd\mcH^n\\
 \leq&C(n,F)R^{-n}\int_{B^{n+}_{R}}W_{f,\tau}z^{n+1}\rd x.
}
Applying Lemma \ref{prop:upper bound of log integral} with $m=n+1$ then readily proves \eqref{eq:uniform-gradient} for $x_0=0$, the general case $x_0\in\Gamma_\varrho$ follows after translation. And an identical argument proves \eqref{eq:uniform-gradient-interior-ball} for the interior ball.
The proof is complete.

\end{proof}

\subsection{Solvability of the mixed boundary problem}

In this subsection we solve the mixed boundary problem by using the barrier method, together with all the estimates we have obtained in previous subsections. First we provide a linear function to give a uniform $L^{\infty}$-bound  for $w_\tau$ as follows.

\begin{lemma}\label{lem:affine}
There is a unique $t_0\in\mbR$ such that
\eq{\label{eq:t0}
 \langle Df(t_0e_1),e_1\rangle=0.
}
Moreover, the affine function $\ell(x)=t_0x_1$
satisfies
\eq{\label{eq:uniform elliptic-PDE one boundary}
 \begin{cases}
 \operatorname{div}(Df_\tau(D\ell))=0 & \text{in }\Omega_\varrho,\\[1mm]
 \langle Df_\tau(D\ell),e_1\rangle=0 & \text{on }\Gamma_\varrho.
 \end{cases}
}

\end{lemma}

\begin{proof}
We first note that the function $t\mapsto\langle Df(te_1),e_1\rangle$ is strictly increasing since
\eq{
 \frac{\rd}{\rd t}\left< Df(te_1),e_1\right>
 =f_{11}(te_1)
 \overset{\eqref{eq:D2f-global}}{\geq}\frac{c(F)}{\sqrt{1+t^2}}\left(1-\frac{t^2}{1+t^2}\right)>0.
}
By the zero-homogeneity of $\bar DF$ and \eqref{eq:<DF(z),z>=F(z)}, we find 
\eq{
 \lim_{t\to-\infty}\langle Df(te_1),e_1\rangle=-F(e_1)<0,
 \quad
 \lim_{t\to+\infty}\langle Df(te_1),e_1\rangle=F(-e_1)>0.
}
Thus we see $t_0$ exists and is unique. Finally, it is obvious that $\ell=t_0x_1$ satisfies \eqref{eq:uniform elliptic-PDE one boundary}, which completes the proof.
\end{proof}

\begin{proposition}\label{prop:height}
For every fixed $0<\tau\leq1$, the minimizer $w_\tau$ satisfies
\eq{\label{eq:height}
 \|w_\tau-\ell\|_{L^\infty(\Omega_\varrho)}\le \sup_{S_\varrho}|\phi-\ell|.
}
\end{proposition}

\begin{proof}
By Proposition \ref{prop:mixed boundary problem} and Lemma \ref{lem:affine},
for any $\zeta\in W^{1,2}(\Omega_\varrho)$ with $\zeta\mid_{S_\varrho}=0$, we have 
\eq{\label{comparison}
    \int_{\Om_\varrho}\left<D f_\tau\left(D w_\tau\right)-D f_\tau(D \ell), D \zeta\right>\rd x
    =0.
}
Taking $\zeta=\zeta_{\tau,+}\coloneqq(w_\tau-\ell-\sup _{S_\varrho}(\phi-\ell))_+\in W^{1,2}$ 
in \eqref{comparison},  we get that $D\zeta_{\tau,+}=0$ almost every in $\Om_\varrho$ thanks to an estimate of $f_\tau$ of the form \eqref{eq:strict-monotone}, which holds since $f_\tau$ is convex.
It follows that from the Poincar\'e inequality that $\zeta_{\tau,+}=0$ and consequently $w_\tau\le \ell+\sup _{S_\varrho}(\phi-\ell)$ a.e. in $\Om_\varrho$.
Similarly, one finds $w_\tau\geq \ell+\inf_{S_\varrho}(\phi-\ell)$ a.e. in $\Om_\varrho$, which completes the proof.

\end{proof}

\begin{proof}[Proof of Theorem \ref{thm:replacement}]

By Proposition \ref{prop:height}, we have uniformly for any $0<\tau\leq1$ that
\eq{\label{eq:rep-height}
 \|w_\tau\|_{L^\infty(\Omega_\varrho)}
 \leq C.
}
Thanks to \eqref{eq:energy-uniform}, we can apply the $BV$ compactness to extract a sequence $\tau_j\searrow0$ such that for some $w\in BV(\Omega_\varrho)$,
\eq{\label{eq:L1-conv}
w_{\tau_j}\to w
 \quad\text{in }L^1(\Omega_\varrho).
}

For any fixed $U\subset\subset\Omega_\varrho\cup\Gamma_\rho$, we cover it by finitely many sets of the following two classes: the first are $B^{n+}_{R/2}(x_0)\cup\Gamma_{R/2}(x_0)$ with $x_0\in\Gamma_\varrho$ and $B^{n+}_{2R}(x_0)\subset\Om_\varrho$, $\Gamma_{2R}(x_0)\subset\Gamma_\varrho$; the second are $B^n_{R/2}(x_0)$ with $B^n_{2R}(x_0)\subset\Om_\varrho$.
On every such set, Theorem \ref{uniform gradient estimate} and \eqref{eq:rep-height} give a gradient bound independent of $\tau${, say $C_0$. By \eqref{eq:D2f-global} there exist positive constants $\lambda_0,\Lambda_0$ independent of $\tau$, such that for any $p\in\mbR^n$ with $\abs{p}\leq C_0$,
\eq{
\lambda_0{\rm I}
\leq D^2f_\tau(p)
\leq\Lambda_0{\rm I},
}
and along the boundary, one has $(f_\tau)_1(Dw_\tau)=0$ and the uniform oblique condition $(f_\tau)_{11}(p)\geq\lambda_0$.
}
In view of the classical regularity theory (cf. \cite{LT86}), we thus obtain
\eq{\label{eq:uniform-local-C2sigma}
 \|w_{\tau_j}\|_{C^{2,\sigma}(U)}\le C_U,
}
for some $\sigma,C_U>0$ independent of $j\in\mbN$.
By a diagonal argument, after passing to a further subsequence (still indexed by $j\in\mbN$),
\eq{\label{eq:C2local-rep}
 w_{\tau_j}\to w
 \quad\text{in }C^2_{loc}(\Omega_\varrho\cup\Gamma_\varrho).
}
We note that the limit agrees with the $L^1$-limit in \eqref{eq:L1-conv}.
Moreover, since $Df_{\tau_j}=(1+\tau_jf)Df$, and the factor $1+\tau_jf(Dw_{\tau_j})$
converges locally uniformly to $1$ as $j\to\infty$, we thus deduce that $w$ satisfies \eqref{eq:replacement-PDE} classically.

We assert that $w\in W^{1,1}(\Omega_\varrho)$. Indeed, let $K_i\subset\subset\Omega_\varrho$ be an exhaustion of $\Om_\varrho$, then by \eqref{eq:L1-conv} we have
\eq{
 \int_{K_i}|Dw|\rd x
 \leq\liminf_{j\to\infty}\int_{\Omega_\varrho}|Dw_{\tau_j}|\rd x
 \le C
}
for some $C>0$ independent of $\tau_j$ and $K_i$, thanks to \eqref{eq:energy-uniform}.
Letting $i\to\infty$ and using the monotonicity convergence yields the assertion.

We next prove \eqref{eq:replacement-trace} by the classical barrier method: we fix
$y\in S_\varrho\setminus E_\varrho$ and choose $r_y>0$ so small that
$B^n_{r_y}(y)\cap\Gamma_\varrho=\emptyset$. Let $\nu_y$ be the outer
normal of $\partial B^n_\varrho$ at $y$.  
For every $x\in\overline{B^n_\varrho}$, it is easy to see that
\eq{\label{eq:sphere-convex}
 -\left<\nu_y,x-y\right>
 =\frac{|x-y|^2}{2\varrho}+\frac{\varrho^2-|x|^2}{2 \varrho}
 \geq \frac{|x-y|^2}{2\varrho}.
}
Since $\phi$ is $C^2$ on the spherical cap,
\eq{\label{ineq:phi-sphereical-bdry}
 \Abs{\phi(x)-\phi(y)-\left<D \phi(y),x-y\right>}
 \leq C_y|x-y|^2
 \quad\text{for any }x\in S_\varrho\cap B_{r_y}(y).
}
For $L>0$ to be determined, we define the affine functions
\eq{
 L_y^\pm(x)
 =\phi(y)+\left< D \phi(y),x-y\right>
 \mp L\left<\nu_y,x-y\right>,
}
which clearly satisfy
\eq{
{\rm div}\left(Df_\tau(DL_y^\pm)\right)
=0.
}

On $S_\varrho\cap B_{r_y}(y)$, in view of \eqref{ineq:phi-sphereical-bdry} and \eqref{eq:sphere-convex}, once we choose $L\geq2\varrho C_y$, then we have
\eq{
L^-_y
\leq\phi
=w_\tau\leq L^+_y.
}
While on $\Om_\varrho\cap\p B_{r_y}(y)$, we have $-\left<\nu_y,x-y\right>\geq\frac{r_y^2}{2\varrho}$ and
\eq{
\Abs{w_\tau(x)-\phi(y)-\left<D\phi(y),x-y\right>}
\leq C_M+\abs{\phi(y)}+\abs{D\phi(y)}r_y,
}
where $C_M\coloneqq\sup_{0<\tau\leq1}\norm{w_\tau}_{L^\infty(\Om_\varrho)}$ is a finite number thanks to \eqref{eq:rep-height}.
Thus after further increasing $L\geq\frac{2\varrho}{r_y^2}\left(C_M+\abs{\phi(y)}+\abs{D\phi(y)}r_y\right)$, we find
\eq{
 L_y^-\le w_\tau\le L_y^+
 \quad\text{on }\p\left(\Om_\varrho\cap B_{r_y}(y)\right).
}
So the weak comparison principle yields $L_y^-\leq w_\tau\leq L_y^+$ in $\Omega_\varrho\cap B_{r_y}(y)$. Passing to the limit along $\tau_j\searrow0$ yields
\eq{\label{ineq:w-L^--L^+-compare}
L_y^-
\leq w\leq L_y^+\quad\text{in }\Omega_\varrho\cap B_{r_y}(y).
}
Since $0\leq-\left<\nu_y,x-y\right>\leq\abs{x-y}$, it follows that for every $x\in \Omega_\varrho\cap B_{r_y}(y)$,
\eq{\label{ineq:w(x)-phi(y)}
|w(x)-\phi(y)| \leq\left(\left|D\phi(y)\right|+L\right) | x-y|.
}
Letting $x\to y$ then shows \eqref{eq:replacement-trace}.

Now we prove \eqref{eq:replacement-weak}.
Let $\zeta$ be a test function with the required properties. For a fixed $\de>0$, we take a cutoff $\eta_\delta$ which$\equiv0$ on $\{x:{\rm dist}(x,S_\varrho)\leq\de\}$, $\equiv1$ on $\{x:{\rm dist}(x,S_\varrho)\geq2\de\}$, and satisfies $\abs{D\eta_\de}\leq C/\de$. Note that $\eta_\delta\zeta$ is admissible in \eqref{eq:weak-ftau}.
By \eqref{eq:C2local-rep}, we have $Df_{\tau_j}(Dw(\tau_j))\to Df(Dw)$.
Passing to the limit in
\eqref{eq:weak-ftau}, we get
\eq{
0
=\int_{\Om_\varrho}\left<Df(Dw),D(\eta_\delta\zeta)\right>\rd x
=\int_{\Om_\varrho} \eta_\de\left<Df(Dw),D\zeta\right>\rd x+\int_{\Om_\varrho}\zeta\left<Df(Dw),D\eta_\de\right>\rd x.
}
Since $Df$ is bounded, $\zeta=0$ on $S_\varrho$, and $|\zeta(x)|\le C{\rm dist}(x,S_\varrho)$ near $S_\varrho$, the first integral on the right converges to $\int_{\Om_\varrho}\left<Df(Dw),D\zeta\right>\rd x$ by dominated convergence, while the second term on the right satisfies
\eq{
\int_{\Om_\varrho}\zeta\left<Df(Dw),D\eta_\de\right>\rd x
\leq\frac{C}\de\int_{\{x\in\Om_\varrho:\de<{\rm dist}(x,S_\varrho)<2\de\}\cap{\rm spt}\zeta}\abs{\zeta}\rd x
\to0\text{ as }\de\searrow0.
}
Thus proving \eqref{eq:replacement-weak}.
The uniqueness is standard to prove and hence omitted, we complete the proof.

\end{proof}

\section{Removability of singularities}\label{sec:4}



\begin{proof}[Proof of Corollary \ref{Prop:removable-singularity}]
Since $K\subset\Om_s\cup\Gamma_s$, there exists $0<\varrho<s$ such that $K\subset\subset \Om_\varrho\cup\Gamma_\varrho$, and $K\cap\overline S_\varrho=\emptyset$.
Consequently, $u\mid_{\overline S_\varrho}$ is $C^2$ and has a $C^2$-extension in a neighborhood of $\overline S_\varrho$, say $\tilde u$.
Now let $w$ be the solution given by Theorem \ref{thm:replacement} in $\Om_\varrho$ with $\phi\coloneqq\tilde u$. 
We note, it suffices to prove that
\eq{
u=w\quad\text{in }\Om_\varrho\setminus K.
}
To prove this, note that $\mcH^{n-1}\left(K\cup E_\varrho\right)=0$ by assumption, and hence for every $\varepsilon>0$ we can find finitely many balls $B^n_{\de_i}(x_i)$ with $x_i\in K\cup E_\varrho$, such that
\eq{
K\cup E_\varrho
\subset\bigcup_{i=1}^{N_\varepsilon}B^n_{\de_i}(x_i),
}
and $\sum_{i=1}^{N_\varepsilon}\de_i^{n-1}<\varepsilon$. 
Then we choose smooth functions $\eta_i$ satisfying $0\leq\eta_i\leq1$ and $\eta_i\equiv0$ on $B^n_{\de_i}(x_i)$, $\equiv1$ outside $B^n_{2\de_i}(x_i)$, as well as $\abs{D\eta_i}\leq\frac{C}{\de_i}$ for some $C$ independent of $i$.
Put $\eta_\varepsilon\coloneqq\prod_{i=1}^{N_\varepsilon}\eta_i$, then $\eta_\varepsilon=0$ in a neighborhood of $K\cup E_\varrho$ and $\eta_\varepsilon\to1$ pointwise in $\Om_\varrho\setminus K$.
Moreover,
\eq{\label{eq:removable-cutoff}
\int_{\Om_\varrho}\abs{D\eta_\varepsilon}\rd x
\leq\sum_{i=1}^{N_\varepsilon}\int_{\Om_\varrho\cap B^n_{2\de_i}(x_i)}\abs{D\eta_i}\rd x
\leq C\sum_{i=1}^{N_\varepsilon}\de_i^{n-1}
\leq C\varepsilon.
}

To proceed, we consider as in \cite{Simon77} the function $\arctan(u-w)$.
Since $\eta_\varepsilon$ vanishes in a neighborhood of $E_\varrho$, there exist a compact set $U_\varepsilon\subset\subset S_\varrho\setminus E_\varrho$ and some $\tilde\de_\varepsilon>0$ such that for any $x\in\Om_\varrho\cap{\rm spt}\eta_\varepsilon$ with ${\rm dist}(x,S_\varrho)<2\tilde\de_\varepsilon$, the nearest point projection of $x$ onto $S_\varrho$, say $y\coloneqq\pi_{S_\varrho}(x)$, satisfies
$y\in U_\varepsilon\subset\subset S_\varrho\setminus E_\varrho$.
Thus we can argue as in the proof of Theorem \ref{thm:replacement}, see in particular \eqref{ineq:w(x)-phi(y)}, to obtain that
\eq{\label{ineq:w(x)-phi(y)-C_varep}
\abs{w(x)-\phi(y)}
\leq\left(\abs{D\phi(y)}+L\right)\abs{x-y}
\eqqcolon\tilde C_\varepsilon{\rm dist}(x,S_\varrho),
}
provided that $L=L(\varepsilon)$ is suitably large.

On the other hand, since $u=\phi$ on $S_\varrho$ and $u$ is $C^2$ in a neighborhood of $U_\varepsilon$, after decreasing $\tilde\de_\varepsilon$ if needed, we obtain
\eq{
\abs{u(x)-\phi(y)}
=\abs{u(x)-u(y)}
\leq\tilde C_\varepsilon\abs{x-y}
=\tilde C_\varepsilon{\rm dist}(x,S_\varrho).
}
Combining with \eqref{ineq:w(x)-phi(y)-C_varep}, we thus find constants
$C_\varepsilon,\delta_\varepsilon>0$ with the property that
\eq{\label{eq:boundary-linear-u-w}
 |\arctan(u(x)-w(x))|
 \le |u(x)-w(x)|
 \le C_\varepsilon{\rm dist}(x,S_\varrho),
}
whenever $x\in\Om_\varrho\cap{\rm spt}\eta_\varepsilon$ and ${\rm dist}(x,S_\varrho)<2\de_\varepsilon$.

For $0<\de<\de_\varepsilon$, we choose another cutoff function $\chi_\de$ such that $\chi_\de(x)=0$ if ${\rm dist}(x,S_\varrho)\leq\de$, $=1$ if ${\rm dist}(x,S_\varrho)\geq2\de,$ and $\abs{D\chi_\de}\leq\frac{C}\de$.
Testing for both $w$ and $u$ the integral identity of the form \eqref{eq:replacement-weak} with $\zeta=\eta_\varepsilon\chi_\de\arctan(u-w)$, and then subtracting, yields
\eq{
&\int_{\Om_\varrho}\frac{\eta_\varepsilon\chi_\de}{1+(u-w)^2}\left<Df(Du)-Df(Dw),Du-Dw\right>\rd x\\
=&-\int_{\Om_\varrho}\chi_\de\arctan(u-w)\left<Df(Du)-Df(Dw),D\eta_\varepsilon\right>\rd x\\
&-\int_{\Om_\varrho}\eta_\varepsilon\arctan(u-w)\left<Df(Du)-Df(Dw),D\chi_\de\right>\rd x.
}
By \eqref{eq:boundary-linear-u-w} and the bound on $D\chi_\de$, we can let $\de\searrow0$ and use a standard dominated convergence argument (as in the proof of Theorem \ref{thm:replacement}) to get
\eq{\label{eq:removable-arctan}
&\int_{\Om_\varrho}\frac{\eta_\varepsilon}{1+(u-w)^2}\left<Df(Du)-Df(Dw),Du-Dw\right>\rd x\\
=&-\int_{\Om_\varrho}\arctan(u-w)\left<Df(Du)-Df(Dw),D\eta_\varepsilon\right>\rd x.
}

By \eqref{eq:removable-cutoff} we see
\eq{
\Abs{\int_{\Om_\varrho}\arctan(u-w)\left<Df(Du)-Df(Dw),D\eta_\varepsilon\right>\rd x}
\leq C(F)\int_{\Om_\varrho}\abs{D\eta_\varepsilon}\rd x
\leq C(F)\varepsilon,
}
and hence
\eq{
0\leq\int_{\Om_\varrho}\frac{\eta_\varepsilon}{1+(u-w)^2}\left<Df(Du)-Df(Dw),Du-Dw\right>\rd x
\leq C(F)\varepsilon.
}
Letting $\varepsilon\searrow0$, by Fatou's Lemma we find
\eq{
\int_{\Om_\varrho\setminus K}\frac{\left<Df(Du)-Df(Dw),Du-Dw\right>}{1+(u-w)^2}\rd x
=0.
}
Hence by \eqref{eq:strict-monotone}, we must have $Du=Dw$ a.e. in $\Om_\varrho\setminus K$.
Using $\mcH^{n-1}(K)=0$, we thus find $u-w=c$ in $\Om_\varrho\setminus K$ for some constant $c$.
Finally, we fix a $y\in S_\varrho\setminus E_\varrho$. In view of \eqref{eq:replacement-trace}, we conclude that
\eq{
\lim_{\Om_\varrho\ni x\to y}w(x)
=\phi(y)
=u(y),
}
and hence the constant above must be $c=0$. This completes the proof.

\end{proof}

\section{Proof of the half-space anisotropic Bernstein theorem}\label{Sec:5}

We divide the proof of Theorem \ref{Thm:Bernstein-aniso} into several Lemmatum. Our starting point is that, by a standard calibration argument we have the following fact:
\begin{lemma}\label{lem:calibration}
Let $u\in C^2(\overline{\mbR^n_+})$ solve \eqref{defn:AMSE} and satisfy \eqref{defn:anisotropic-free-bdry-Intro}. Then its subgraph
\eq{
E_u
\coloneqq\{(x,t)\in\mbR^{n+1}_+:t<u(x)\}
}
is an $F$-minimizer in $(A,\mbR^{n+1}_+)$ for every open set
$A\subset\mbR^{n+1}$.
\end{lemma}

\begin{lemma}\label{lem:SSA77-Thm3.1}
Let $n\geq2, r>0,$ and let $E\subset\mbR^{n+1}_+$ be a $F$-minimizer in $(B_r,\mbR^{n+1}_+)$.
Set $\mcS_r(E)\coloneqq B_r\cap\overline{\mbR^{n+1}_+\cap\p^\ast E}$.
Let $\mcR_r(E)$ be the set of points at which $\mcS_r(E)$ is a $C^2$-hypersurface, possibly with boundary contained in $\p\mbR^{n+1}_+$, and put $\S_r(E)\coloneqq\mcS_r(E)\setminus\mcR_r(E)$.
Let $\{\mcM_\alpha\}_{\alpha\in I}$ denote the connected components of $\mcR_r(E)$.
If
\eq{
\mcH^{n-1}(\S_r(E))=0,
}
then the following statements hold.
\begin{enumerate}
    \item For every $\alpha$ there is a $F$-minimizer $E_\alpha\subset \mbR^{n+1}_+$ in $(B_r,\mbR^{n+1}_+)$ such that, as vector Radon measures in $B_r\cap \mbR^{n+1}_+$,
\eq{\label{eq:component-GG-measure}
D\chi_{E_\alpha}
=D\chi_E\llcorner\mcM_\alpha.
}
In particular, the outer normal of $E_\alpha$ agrees with that of $E$
on $\mcM_\alpha$, and
\eq{\label{eq:component-support-closure}
\mcS_r(E_\alpha)
={\rm cl}_{B_r\cap\overline{\mbR^{n+1}_+}}(M_\alpha)
}
    \item The Gauss--Green and perimeter measures decompose as
\eq{\label{eq:component-GG-decomposition}
D\chi_E\llcorner(B_r\cap\mbR^{n+1}_+)
=\sum_{\alpha\in I}D\chi_{E_\alpha},
\quad
|D\chi_E|\llcorner(B_r\cap\mbR^{n+1}_+)
=\sum_{\alpha\in I}|D\chi_{E_\alpha}|.
}
    \item The family $\left\{{\rm cl}_{B_r\cap\overline{\mbR^{n+1}_+}}(M_\alpha)\right\}_{\alpha\in I}$
is locally finite in $B_r$.
\end{enumerate}
\end{lemma}
\begin{proof}
The first assertion is essentially proved in \cite[Corollary 3.1]{SSA77}, and the second and the third assertions are essentially \cite[Theorem 3.1]{SSA77}.
\end{proof}

In all follows, we continue to use the notations in Lemma \ref{lem:SSA77-Thm3.1}.
From \cite{PM15}, we have the following Ahlfors regularity.

\begin{lemma}
Let $n\geq2, r>0,$ and let $E\subset\mbR^{n+1}_+$ be a $F$-minimizer in $(B_r,\mbR^{n+1}_+)$.
There exist positive constants $c_F=c_F(n,F), C_F=C_F(n,F)$ such that, for every $X\in\mcS_r(E)$ and every $0<\rho<\frac{1}{2}{\rm dist}(X,\p B_r)$, one has
\eq{\label{ineq:Ahlfors-regularity}
c_F\rho^n
\leq P_F(E;B_\rho(X)\cap\mbR^{n+1}_+)
\leq C_F\rho^n.
}
\end{lemma}

The following Lemma in the spirit of Simon \cite{Simon77} is the key ingredient to prove Theorem \ref{Thm:Bernstein-aniso}.

\begin{lemma}\label{lem:Simon77-key}
Let $n=3$, and let $E\subset\mbR^{4}_+$ be a $F$-minimizer in $(B_2,\mbR^{4}_+)$, with
$0\in\mcS_2(E)$. If $\mcH^1(\S_2(E))=0$ and
\eq{\label{condi:D_4-chi_E-<=0}
D_4\chi_E\leq0\text{ in }B_2\cap\mbR^4_+
}
in the sense of distributions, then we have $0\in\mcR_1(E)$.
\end{lemma}

\begin{proof}
First we note, by \eqref{condi:D_4-chi_E-<=0} and \eqref{defn:reduced-bdry-nu_E},
\eq{
\nu_4
\coloneqq\left<\nu,e_4\right>
\geq0\text{ on }\mcR_2(E).
}
For any connected component $M_\ast$ of $\mcR_2(E)$, we define $V_F(x)\coloneqq\frac{\nu_4(x)}{F(\nu(x))}\geq0$ on $M_\ast$.
In view of Lemmas \ref{Lem:V_F-PDE}, \ref{Lem:tangential-property-bdry} we see, $V_F$ satisfies the equation on $\mcM_\ast$:
\eq{\label{eq:De_F-W_f}
\Delta_F V_F
=&-2 g\left(\nabla \log F(\nu), \nabla_F\left(V_F\right)\right)-\frac{V_F}{F(\nu)} \operatorname{tr}_g\left(h_F^2\right),
}
together with the boundary condition
$g\left(\nabla_F V_F, \mu\right) \equiv 0$ along the regular part of the boundary.
Hence on $M_\ast$, if we consider the quantity (which is a weighted $F$-Laplacian)
\eq{
{\rm div}_{\mcM_\ast}\left(F^2(\nu)\na_F\phi\right)
=F^2(\nu)\left(\De_F\phi+2g\left(\na\log F(\nu),\na_F\phi\right)\right),
\quad\forall\phi\in C^2(\mcM_\ast),
}
then in view of \eqref{eq:De_F-W_f} we find
\eq{
{\rm div}_{\mcM_\ast}\left(F^2(\nu)\na_F(V_F)\right)
+F(\nu)V_F{\rm tr}_g(h^2_F)
=0.
}
This gives a uniformly elliptic equation for $V_F$ on $\mcM_\ast$, hence we have the strong Harnack inequality (cf., \cite[Theorem 5.44]{Lieberman13})
\eq{
\inf_{\mcB^g_{\widetilde\rho_{X}}(X)} V_F
\geq C\sup_{\mcB^g_{\widetilde\rho_{X}}(X)}V_F
}
for any $X\in\mcM_\ast$ some suitably small $\tilde\rho_{X}>0$,
here we denote by $g$ the induced metric of $\mcM_\ast\subset\mbR^4$, by ${\rm dist}_g$ the induced distance on $(\mcM_\ast,g)$, and by $$\mcB^g_r(X)\coloneqq\left\{Y\in M_\ast:{\rm dist}_g(Y,X)<r\right\}$$ the \textit{geodesic ball} on $(\mcM_\ast,g)$ with radius $r$ centered at $X$.
Since $0<\min_{\xi\in\mbS^3}F(\xi)\leq F(\nu)\leq \max_{\xi\in\mbS^3}F(\xi)$, this in particular gives us the \textit{dichotomy}: either $\nu_4>0$ on $\mcM_\ast$, or $\nu_4\equiv0$ on $\mcM_\ast$.
We point out that the dichotomy was also shown in \cite{Simon77}, but our proof is slightly different from the original proof, since we have to find a suitable PDE to incorporate with the anisotropic free boundary condition.

We now divide the proof into several steps.

\noindent{\bf Step 1. We show that the closure of each connected component is smooth. Namely, for any $Y_0\in\S_1(E)$, if $Y_0\in\overline{\mcM_\ast}$ for some
connected component $\mcM_\ast$, then $\overline{\mcM_\ast}$ is a $C^2$-hypersurface near $Y_0$.
}

In view of the dichotomy, we break into two cases.

\noindent{\em Case 1. $\nu_4>0$ on $\mcM_\ast$.}

We start with the following claim, which is due to De Giorgi \cite[Lemma II]{DeGiorgi65}.

{\em Claim 1. For any $h>0$ small, let $B_{1,h}^+\coloneqq\{X\in B_1^+: X+[0,h]e_4\subset B_1^+\}$, then
\eq{
\chi_{E_\ast}(X+he_4)
\leq\chi_{E_\ast}(X)\text{ for a.e. }X\in B_{1,h}^+, 
}
where $E_\ast$ is the $F$-minimizer associated to $\mcM_\ast$ by Lemma \ref{lem:SSA77-Thm3.1}.
}

\medskip

To see this, we first note that by \eqref{eq:component-GG-measure} and \eqref{condi:D_4-chi_E-<=0}, we have
\eq{\label{ineq:D_4-chi_E_ast}
D_4\chi_{E_\ast}
=(D_4\chi_E)\llcorner\mcM_\ast
\leq0.
}
Then we choose any non-negative test function $\varphi\in C^\infty_c(B^+_{1,h})$ and for $h>s>0$, put
\eq{
I(s)
\coloneqq\int_{B^+_{1,h}}\chi_{E_\ast}(X+se_4)\varphi(X)\rd\mcL^4(X)
=\int_{B_{1,h}^+}\chi_{E_\ast}(Y)\varphi(Y-se_4)\rd\mcL^4(Y).
}
Hence by \eqref{ineq:D_4-chi_E_ast},
\eq{
I'(s)
=-\int_{B^+_{1,h}}\chi_{E_\ast}(Y)\p_4\varphi(Y-se_4)\rd\mcL^4(Y)
=\left<D_4\chi_{E_\ast},\varphi(\cdot-se_4)\right>
\leq0,
}
which implies that $I(h)\leq I(0)$ for any small $h>0$, namely,
\eq{
\int_{B_{1,h}^+}\left(\chi_{E_\ast}(X+he_4)-\chi_{E_\ast}(X)\right)\varphi(X)\rd\mcL^4(X)
\leq0.
}
Since $\varphi$ is arbitrary, {\em Claim 1} is thus proved.

Denote the (relative) measure-theoretic interior and exterior of $E_\ast$ as
\eq{
{\rm Int}_{\mbR^4_+}(E_\ast)
\coloneqq&\left\{X\in\mbR^4:\text{there exists }\rho>0\text{ s.t. }\Abs{B_\rho(X)\cap\mbR^4_+\setminus E_\ast}=0\right\},\\
{\rm Ext}_{\mbR^4_+}(E_\ast)
\coloneqq&\left\{X\in\mbR^4:\text{there exists }\rho>0\text{ s.t. }\Abs{B_\rho(X)\cap\mbR^4_+\cap E_\ast}=0\right\}.
}

If for some $h>0$ and some $X\in B^+_{1,h}$ we have $X+he_4\in{\rm Int}_{\mbR^4_+}(E_\ast)$, then by definition there exists $\rho>0$ such that for $\mcL^4$-a.e. $Z\in B_\rho(X)\cap\mbR^4_+$, we have $Z+he_4\in E_\ast$.
By virtue of {\em Claim 1}, we find
\eq{
1=\chi_{E_\ast}(Z+he_4)
\leq\chi_{E_\ast}(Z),
}
implying that $Z\in E_\ast$ for $\mcL^4$-a.e. $Z\in B_\rho(X)\cap\mbR^4_+$, and hence $X\in{\rm Int}_{\mbR^4_+}(E_\ast)$.
A similar property holds for ${\rm Ext}_{\mbR^4_+}(E_\ast)$. Combining, we have thus obtained: for small $h>0$ and $X\in B^+_{1,h}$ the implications hold:
\eq{\label{impli:Int-Ext}
X+he_4\in{\rm Int}_{\mbR^4_+}(E_\ast)
\Longrightarrow&X\in{\rm Int}_{\mbR^4_+}(E_\ast)\\
X\in{\rm Ext}_{\mbR^4_+}(E_\ast)\Longrightarrow&X+he_4\in{\rm Ext}_{\mbR^4_+}(E_\ast).
}

{\em Claim 2. For $Y_0$ as in the statement of {\bf Step 1}, write $Y_0=(y_0,s_0)$ for some $y_0\in\overline{\mbR^3_+}$ and $s_0\in\mbR$, and let $\pi:\overline{\mbR^4_+}\to\overline{\mbR^3_+}$ be the projection mapping. Denote for simplicity $\mcS_\ast\coloneqq{\rm cl}_{B_3\cap\overline{\mbR^4_+}}(\mcM_\ast)$. Then for any sufficiently small $\varrho>0$,  $y_0$ is a relative interior point of $\pi(\mcS_\ast\cap\mathcal Q_\varrho)$ in $\overline{\mbR^3_+}$, where $\mathcal Q_\varrho\coloneqq\overline{B^{3+}_\varrho(y_0)}\times[s_0-\varrho,s_0+\varrho]$ is a closed cylinder in $B^+_2$.}

\medskip

We fix a sufficiently small $\varrho>0$. Arguing by contradiction, we assume that $y_0$ is not a relative interior point of $\pi(\mcS_\ast\cap\mathcal Q_\varrho)$ in $\overline{\mbR^3_+}$.
Hence there exists a sequence of points $\{y_j\}_{j\in\mbN}$, with $y_j\to y_0$ as $j\to\infty$ but $y_j\notin\pi(\mcS_\ast\cap\mathcal Q_\varrho)$ for each $j$.
From this we infer that, for sufficiently large $j$,
\eq{
L_j\coloneqq\{y_j\}\times[s_0-\frac{\varrho}{2},s_0+\frac{\varrho}{2}]
\subset\mathcal{Q}_\varrho\setminus\mathcal{S}_\ast.
}
It follows that we must have either $L_j\subset{\rm Int}_{\mbR^4_+}(E_\ast)$ or $L_j\subset{\rm Ext}_{\mbR^4_+}(E_\ast)$ for each sufficiently large $j\in\mbN$.
After passing to a subsequence (still indexed by $j$), we may assume that $L_{j}\subset{\rm Int}_{\mbR^4_+}(E_\ast)$ for all $j\in\mbN$ or $\subset{\rm Ext}_{\mbR^4_+}(E_\ast)$ for all $j$.
Here we consider only the former case, and the later follows similarly.

We fix an $s\in(s_0,s_0+\frac{\varrho}2)$.
Since $y_j\to y_0$ and $(y_j,s)\in{\rm Int}_{\mbR^4_+}(E_\ast)$, we have
\eq{\label{inclu:y_0,s-int-closure}
(y_0,s)\in{\rm cl}_{\overline{\mbR^4_+}}\left({\rm Int}_{\mbR^4_+}(E_\ast)\right).
}
On the other hand, since $Y_0=(y_0,s_0)\in\mcS_\ast$, there exists a sequence of points $\{Z_k=(z_k,t_k)\}_{k\in\mbN}$ with $Z_k\in{\rm Ext}_{\mbR^4_+}(E_\ast)$ and $Z_k\to Y_0$ as $k\to\infty$.
Moreover, since $s>s_0$, for sufficiently large $k$ we must have $t_k<s$.
Using \eqref{impli:Int-Ext} for $Z_k=(z_k,t_k)\in{\rm Ext}_{\mbR^4_+}(E_\ast)$, we thus find $(z_k,s)\in{\rm Ext}_{\mbR^4_+}(E_\ast)$.
Letting $k\to\infty$ yields
\eq{\label{inclu:y_0,s-ext-closure}
(y_0,s)\in{\rm cl}_{\overline{\mbR^4_+}}\left({\rm Ext}_{\mbR^4_+}(E_\ast)\right).
}
Combining \eqref{inclu:y_0,s-int-closure} with \eqref{inclu:y_0,s-ext-closure} and taking \eqref{ineq:Ahlfors-regularity} into account, we see $(y_0,s)\in\mcS_\ast$. Since $s\in(s_0,s_0+\frac{\varrho}2)$ is arbitrary, this implies
\eq{
l_0\coloneqq\{y_0\}\times[s_0,s_0+\frac{\varrho}4]\subset\mcS_\ast
\subset\overline{\mcM_\ast}.
}
Since $\overline{\mcM_\ast}\setminus\mcM_\ast\subset\S_2(E)$ and $\mcH^1(\S_2(E))=0$, we see that $\mcH^1$-a.e. $X\in l_0$ must be an interior regular point of $\mcM_\ast$. However, at any such point $e_4$ must be a tangent vector to $\mcM_\ast$, and hence $\nu_4=0$, contradicting $\nu_4>0$.
This proves {\em Claim 2}.
As a by-product, we have also shown that: there exists a suitably small $0<\varsigma<\varrho$ such that
\eq{
\mcM_\ast\cap\left(\left(B^{3+}_\varsigma(y_0)\cap\mbR^3_+\right)\times\mbR\right)
={\rm graph}(v),
}
where $v\in C^2(\overline{B_\varsigma^{3+}(y_0)}\setminus\mathcal K)$ for $\mathcal K\coloneqq\pi\left((\overline{\mcM_\ast}\setminus\mcM_\ast)\cap\mathcal Q_\varrho\right)$. 
Note that $\mcH^1(\mathcal K)=0$, thanks to $\mcH^{1}(\S_2(E))=0$.

If $Y_0\in\mbR^4_+$, then $y_0\in\mbR^3_+$ is a (interior) removable singularity thanks to \cite{Simon77}.
If $Y_0\in\p\mbR^4_+$, then $y_0\in\p\mbR^3_+$ is a free boundary point, and by $\mcH^1(\mathcal K)=0$ we may choose $0<s<\varsigma$ such that $\mathcal K\cap\p B^3_s(y_0)=\emptyset$. Set $K\coloneqq\mathcal K\cap\overline{B^{3+}_s(y_0)},$ then we can apply the removability of singularities (Corollary \ref{Prop:removable-singularity}) to deduce that $v$ can be extended across $K$ as a $C^2$ solution. Hence we conclude $\overline{\mcM_\ast}$ is smooth near $Y_0$ in both cases.

\

\noindent{\em Case 2. $\nu_4\equiv0$ on $\mcM_\ast$.}

In this case, since $\nu_4\equiv0$ on $\mcM_\ast$ we have
\eq{
D_4\chi_{E_\ast}
=0
}
for a $F$-minimizer $E_\ast\subset\mbR^4_+$ in $(B_2,\mbR^4_+)$, thanks to \eqref{eq:component-GG-measure}.
By \cite[Lemma 28.13 (ii)]{Mag12}, for any sufficiently small $\varrho>0$, there is a set $G_\ast\subset\mbR^3_+$ of locally finite perimeter, such that
\eq{
E_\ast\cap\mathcal Q_\varrho
=G_\ast\times(s_0-\varrho,s_0+\varrho),
}
where we have used the notations introduced in {\em Claim 2} above.
In particular, if $\overline{\mcM_\ast}$ were not smooth at $Y_0$, then there would be a vertical line segment in $\overline{\mcM_\ast}\setminus\mcM_\ast\subset\S_2(E)$ through $Y_0$, contradicting the assumption that $\mcH^{1}(\S_2(E))=0$.
This completes {\bf Step 1}.

\noindent{\bf Step 2. We prove that two distinct connected component closures may not pass through the same singular point $Y_0\in\S_1(E)$.}

We argue by contradiction and suppose that $Y_0\in\overline{\mcM_\ast}\cap\overline{\mcM_\star}$ for two distinct connected components $\mcM_\ast$ and $\mcM_\star$.
By {\bf Step 1}, we know both closures are smooth near $Y_0$.
In view of the dichotomy established before {\bf Step 1}, we break into three subcases.

First, if $\nu_4>0$ on both components, then they are graphs of $v_\ast$, $v_\star$ over the same domain in $\mbR^3_+$.
We note that the two graphs cannot intersect transversely, otherwise producing a $2$-dimensional submanifold, which would belong to $\S_2(E)$, a contradiction to $\mcH^1(\S_2(E))=0$.
Thus the graphs meet tangentially, so their gradients agree at each contact point.
Consequently, the functions
\eq{
v_+
=\max\{v_\ast,v_\star\},\quad
v_-
=\min\{v_\ast,v_\star\}
}
are $C^{1,1}$-weak solutions as in \cite[Lemma 2.4]{SSA77}, and the difference $w=v_+-v_-$ satisfies
\eq{
\begin{cases}
\div(\mcA(x)Dw)=0,\\
\langle\mcA(x)Dw,e_1\rangle=0
\quad\text{on the possibly non-empty boundary},
\end{cases}
}
where $\mcA(x)=\int_0^1D^2f\left(Dv_-(x)+t(Dv_+(x)-Dv_-(x))\right)\rd t$, determined by $F,v_\ast,v_\star$, is positive definite.
Using again the strong Harnack inequality (cf., \cite[Theorem 5.44]{Lieberman13}), we conclude that $w\equiv0$, which contradicts the fact that $\mcM_\ast$ and $\mcM_\star$ are distinct.

Second, if $\nu_4\equiv0$ on both components, then by the proof of {\em Step 1, Case 2}, there exists a vertical line segment $l\in\overline{\mcM_\ast}\cap\overline{\mcM_\star}$.
If any $Y\in l$ is a regular point of $\mcS_2(E)$, then
\eq{\label{eq:mcM_ast-mcM_star}
Y\in\mcM_\ast\cap\mcM_\star.
}
Here we have used the topological fact that, connected components are relatively closed, i.e. $\mcM_\ast={\rm cl}_{\mcR_2(E)}(\mcM_\ast)$, and similarly for $\mcM_\star$.
However, since both $\mcM_\ast$, $\mcM_\star$ are connected components, \eqref{eq:mcM_ast-mcM_star} would imply that $\mcM_\ast=\mcM_\star$, a contradiction to the assumption that these two components are distinct.
Thus we have shown that the whole line segment $l\subset\S_2(E)$, and this is again a contradiction to $\mcH^1(\S_2(E))=0$. Therefore, the assumption that $\mcM_\ast$, $\mcM_\star$ are distinct cannot hold in this case.

Finally, suppose $\nu_4>0$ on $\mcM_\ast$ and $\nu_4\equiv0$ on $\mcM_\star$.
By the proof of {\em Step 1} we know that $\mcM_\ast$ is a graph over $\mbR^3_+$ locally around $Y_0$, while $\mcM_\star$ is a cylinder locally around $Y_0$, therefore $\mcM_\ast$ must meet $\mcM_\star$ transversally around $Y_0$, which would produce a $2$-dimensional submanifold belonging to $\S_2(E)$, thus a contradiction to $\mcH^1(\S_2(E))=0$.
Combining, {\bf Step 2} is thus completed.

\noindent{\bf Step 3. We show that $0\in\mcR_1(E)$.}

Note that $\mcS_1(E)=\overline{\mcR_1(E)}$, thanks to \eqref{ineq:Ahlfors-regularity}.
Hence there exists $Y_j\in\mcR_1(E)$ such that $Y_j\to0$ as $j\to\infty$.
By virtue of Lemma \ref{lem:SSA77-Thm3.1} (3), after extracting a subsequence (still indexed by $j\in\mbN$), we could assume that $Y_j\in\mcM_\ast$ for some connected component $\mcM_\ast$ and for all $j\in\mbN$.
Thus $0\in\overline{\mcM_\ast}$.
Combining with {\em Step 1} and {\em Step 2}, we thus conclude that $0\in\mcR_1(E)$. This completes the proof.
\end{proof}

We are now ready to prove the Bernstein theorem.

\begin{proof}[Proof of Theorem \ref{Thm:Bernstein-aniso}]
Let $M={\rm graph}(u)$, and let $E=\left\{(x,t)\in\mbR^4_+:t<u(x)\right\}$ be the subgraph.
By Lemma \ref{lem:calibration}, $E$ is an $F$-minimizer in $(A,\mbR^4_+)$ for every open $A\subset\mbR^4$.
We also assume WLOG that $0\in\overline{M}\cap\p\mbR^4_+$.

Now we choose any blow-down sequence $r_j\to\infty$ as $j\to\infty$ and put
\eq{
E_j
=r_j^{-1}E,\quad
M_j
=r_j^{-1}M.
}
Note that every $E_j$ is also a $F$-minimizer and $0\in\mcS_2(E_j)$, where we have adopted the notations in Lemma \ref{lem:SSA77-Thm3.1}.
Applying the compactness result of De Philippis-Maggi \cite[Theorem 2.9]{PM15}, after extracting a subsequence (still indexed by $j\in\mbN$), we have
\eq{
\chi_{E_j}
\to\chi_{E_{\infty}}\text{ in }L^1_{loc}(B_2)
}
for some $F$-minimizer $E_\infty\subset\mbR^4_+$ in $(B_2,\mbR^4_+)$.
Moreover, locally in $B_2\cap\mbR^4_+$ we have
\eq{
D\chi_{E_j}\wsc D\chi_{E_\infty},\quad
\abs{D\chi_{E_j}}\wsc\abs{D\chi_{E_\infty}},
}
and that $0\in\mcS_2(E_\infty)$.
We also note that, since each $E_j$ is a subgraph, we have $D_4\chi_{E_j}\leq0$ in the distributional sense, and hence passing to the limit, we have
\eq{
D_4\chi_{E_\infty}
\leq0.
}
On the other hand, by the interior regularity \cite{SSA77} and boundary regularity results \cite{PM15,PM17}, we have
\eq{
\mcH^1(\S_2(E_\infty))=0.
}
Combining, we see that Lemma \ref{lem:Simon77-key} is applicable, and yields that $0$ is a regular boundary point of $\mcS_2(E_\infty)$.
We denote by $\nu_\infty(0)$ the outer unit normal of $\mcS_2(E_\infty)$ at $0$.
Clearly, there exists $r>0$ sufficiently small such that
\eq{
{\rm exc}^H_{\nu_\infty}(E_\infty,0,2r)
<\frac{1}{2}\varepsilon_{{\rm crit}},
}
and $\abs{D\chi_{E_\infty}}(\p B_{2r})=0$ (indeed, this holds for a.e. $r>0$, cf. \cite[Proposition 2.16]{Mag12}).
Here ${\rm exc}^H_{\nu_\infty}(E_\infty,0,2r)$ is the so-called {\em cylindrical excess} in \cite{PM15} and $\varepsilon_{\rm crit}=\varepsilon_{\rm crit}(n,F)$ is the constant from the $\varepsilon$-regularity result \cite[Theorem 3.1]{PM15}.
By a standard continuity argument (see \cite[Remark 3.6]{PM15}), we have
\eq{
{\rm exc}_{\nu_\infty(0)}^H(E_j,0,2r)
\to
{\rm exc}_{\nu_\infty(0)}^H(E_\infty,0,2r),
}
and it follows that
\eq{
{\rm exc}^H(E_j,0,2r)
\leq{\rm exc}^H_{\nu_\infty(0)}(E_j,0,2r)
<\varepsilon_{\rm crit}
}
for all sufficiently large $j$, where ${\rm exc}^H(E_j,0,2r)$ is the {\em spherical excess} in \cite{PM15}.
Therefore, the boundary $\varepsilon$-regularity theorem \cite[Theorem 3.1]{PM15} then gives uniform $C^{1,1/2}$ graphicality of $M_j\cap B_{cr}$ over the tangent hyperplane $T_0(\mcS_2(E_\infty))$ for some $c=c(F)>0$.
Moreover, from the $C^{1,\frac12}$-estimate in \cite[Theorem 3.1]{PM15} we see, for any $Y_j\in M_j\cap B_{cr}$ with $Y_j\to0\in\mcS_2(E_\infty)\cap B_{cr}$, one has
\eq{\label{ineq:normal-convergence}
\nu_{E_j}(Y_j)
\to\nu_\infty(0).
}

Now for any $X\in M$, we have $X_j\coloneqq r_j^{-1}X\in M_j$ and $X_j\to0$, also $X_j\in B_{cr}$ for all sufficiently large $j$. By \eqref{ineq:normal-convergence}, we have $\lim_{j\to\infty}\nu_{E_j}(X_j)=\nu_\infty(0)$.
Moreover, note that $\nu_{E_j}(X_j)=\nu_E(X)$ if we rescale back, which in turn implies that
\eq{
\nu_\infty(0)
=\lim_{j\to\infty}\nu_{E_j}(X_j)
=\lim_{j\to\infty}\nu_E(X)
=\nu_E(X).
}
By the arbitrariness of $X$, we conclude that $M$ is a hyperplane.
This completes the proof.
\end{proof}

\bibliography{BibTemplate.bib}
\bibliographystyle{amsalpha}
\end{document}